\documentclass[11pt,reqno]{amsart}
\usepackage{amscd,amsmath,amsopn,amssymb,amsthm,multicol}
\usepackage[backref=page, breaklinks=true,colorlinks=true,linkcolor=blue,citecolor=blue,urlcolor=blue]{hyperref}

\usepackage{enumerate,stmaryrd}

\def\ip{\langle \cdot\,,\cdot\rangle}

\DeclareMathOperator{\ad}{ad}
\DeclareMathOperator{\Id}{Id}

\DeclareMathOperator{\spann}{span}

\DeclareMathOperator{\Der}{Der}
\DeclareMathOperator{\Cl}{Cl}

\renewenvironment{proof}[1][Proof]{\textbf{#1.} }
{\ \rule{0.5em}{0.5em}}
\newtheorem{theorem}{Theorem}
\newtheorem{prop}{Proposition}
\newtheorem{lemma}{Lemma}

\theoremstyle{definition}
\newtheorem{definition}{Definition}
\newtheorem{remark}{Remark}

\begin{document}

\title
[Geodesic orbit pseudo-Riemannian nilmanifolds] {Geodesic orbit pseudo-Riemannian $H$-type nilmanifolds}

\author[K.~Furutani, I.~Markina, Yu.G.~Nikonorov]{Kenro Furutani, Irina Markina, and Yurii Nikonorov}

\address{K.~Furutani. Osaka Central Advanced Mathematical Institute, Osaka Metropolitan University,
Sugimoto, Sumiyoshi-ku, Osaka 558-8585, Japan}
\email{kf46089@gmail.com}

\address{I.~Markina. Department of Mathematics, University of Bergen, P.O.~Box 7803,
Bergen N-5020, Norway}
\email{irina.markina@uib.no}

\address{Yu.G.~Nikonorov. Southern Mathematical Institute of VSC RAS,
53 Vatutina St., Vladikavkaz, 362025, Russia}
\email{nikonorov2006@mail.ru}

\begin{abstract}
A pseudo-Riemannian manifold is called geodesic orbit (GO) if every geodesic is an orbit of a one-parameter group of isometries. For Riemannian $H$-type groups, the GO property was completely characterized by A. Kaplan and C. Riehm in terms of the dimension of the centre and the structure of the underlying Clifford module. We solve the corresponding problem for pseudo $H$-type groups $N_{r,s}$, that is, 2-step nilpotent Lie groups whose Lie algebra $\mathfrak n_{r,s}=\mathfrak z\oplus\mathfrak v$ is built from an admissible module $\mathfrak v$ of the Clifford algebra $\mathrm{Cl}(\mathbb R^{r,s})$, endowed with the left-invariant metric whose restriction to the centre is indefinite ($s\ge 1$). We prove that $N_{r,s}$ is naturally reductive if and only if $(r,s)\in\{(0,1),(1,2)\}$, and then it is GO for every admissible module. We also prove that $N_{3,4}$ is GO, but not naturally reductive, exactly when $\mathfrak v$ is a minimal admissible module, and that in all remaining cases $N_{r,s}$ is not GO. Thus, in contrast with the Riemannian situation, where a seven-dimensional centre admits GO metrics on isotypic modules of dimensions 8, 16 and 24, the exceptional pseudo-Riemannian case is GO only for the smallest possible module. For $N_{3,4}(\mathfrak v_{\min})$ we explicitly construct the one-parameter isometry groups generating all geodesics, including the null ones, and describe the affine family of their generators.
\vspace{2mm}


\vspace{2mm} \noindent Key words and phrases: geodesic orbit Riemannian manifolds, geodesic orbit pseudo-Riemannian manifolds,  geodesic orbit spacec, $H$-type Lie groups
\end{abstract}

\maketitle

\tableofcontents

\today
\section{Introduction and main results}

A pseudo-Riemannian manifold $(M,g)$ is called a geodesic orbit manifold (GO manifold) if every geodesic of $M$ is an orbit of a one-parameter subgroup of the full isometry group of $(M,g)$. Such a manifold is necessarily homogeneous, $M=G/H$. The terminology goes back to O. Kowalski and L. Vanhecke, who initiated the systematic study of Riemannian GO manifolds~\cite{KV}. Symmetric spaces, weakly symmetric spaces, naturally reductive spaces and normal homogeneous spaces all belong to this class. For overviews of the Riemannian theory we refer to~\cite{Arv17,BerNik20, CNN2023,Du1,Nik2017}.

In the pseudo-Riemannian setting the GO property is considerably more subtle. A geodesic may be null, and for null geodesics the classical geodesic lemma must be modified: the generator of the one-parameter group is only required to reproduce the geodesic up to an affine reparametrization, see Dušek and Kowalski~\cite{DK2007}. This additional freedom, encoded in the constant $k$ in Proposition~\ref{pr.geodlem1} below, has no Riemannian counterpart. It is one reason why many Riemannian results do not carry over directly. Pseudo-Riemannian GO spaces, in particular Lorentzian ones and nilmanifolds, have recently attracted considerable attention; see  \cite{Bar, DBO,CNWZ2024,CWZ2022, DK2007, DK2007n, Du2009, NW23,WC22}. Nilmanifolds are a natural testing ground, since their full isometry group is explicitly known, see Proposition~\ref{pr.isgr.1}. The GO property then reduces to a linear-algebraic problem about skew-symmetric derivations.

Groups of Heisenberg type (shortly $H$-type groups), introduced by A. Kaplan~\cite{Kap80}, form one of the most important classes of 2-step nilpotent Lie groups. They appear in harmonic analysis, in sub-Riemannian geometry and in the construction of Damek–Ricci harmonic spaces. Their Lie algebras $\mathfrak n=\mathfrak z\oplus\mathfrak v$ are determined by Clifford modules: the map $Z\mapsto J_Z$, dual to the adjoint representation, makes $\mathfrak v$ a module over the Clifford algebra $\mathrm{Cl}(\mathfrak z)$. Allowing the scalar product on $\mathfrak z$ to be indefinite leads to pseudo $H$-type groups $N_{r,s}$~\cite{Cia}. Here $\mathfrak z\cong\mathbb R^{r,s}$, and $\mathfrak v$ is an admissible module of $\mathrm{Cl}(\mathbb R^{r,s})$ carrying a neutral scalar product. Their classification up to isomorphism and isometry, and their automorphism groups, were obtained in~\cite{FuMa17,FuMa19,FuMa21}. Compared with the Riemannian case the algebraic structure is much richer.

For Riemannian $H$-type groups the GO problem was solved by Kaplan~\cite{Kap83} and Riehm~\cite{Rie84}.

\begin{theorem}[\cite{Kap83,Rie84}]\label{th.class.riemmannain}
Let $N$ be an $H$-type group supplied with a left invariant Riemannian metric
with the $H$-type algebra $\mathfrak{n}=\mathfrak{z}\oplus \mathfrak{v}$, where $\mathfrak z$ is the $m$-dimensional centre and $\mathfrak v$ is $n$-dimensional orthogonal complement.
Then $N$ is geodesic orbit if and only if one of the following three conditions holds:

{\rm 1)} $m = 1, 2, 3$ and $n$ is any possible;

{\rm 2)} $m = 5, 6$ and $n = 8$;

{\rm 3)} $m = 7$, $n =8, 16, 24$ and $\mathfrak{v}$ is an isotypic Clifford module {\rm(}in this case it is equivalent to the following property:
if $Z_1,Z_2,\dots,Z_7$ is an orthonormal basis of $\mathfrak{z}$,
the linear transformation $X \mapsto J_{Z_1}J_{Z_2}\cdots J_{Z_7}(X)$ of $\mathfrak{v}$ is either $\Id$ or $-\Id${\rm)}.

Moreover, $N$ is naturally reductive if and only if $m=1$ or $m=3$.
\end{theorem}

In~\cite{FuMaNik25,MNF26} we began the study of the pseudo-Riemannian case under the assumption that $\mathfrak v$ is a minimal admissible module. In the present paper we remove this assumption and obtain a complete answer for arbitrary admissible modules.

\begin{theorem}\label{th.class.pseudo}
Let $N_{r,s}$ be a pseudo $H$-type group supplied with the left invariant pseudo-Riemannian metric, where $(r,s)$, $s\geq 1$, is the signature of restriction of the metric to the centre. Let $\mathfrak{n}_{r,s}=\mathfrak{z}\oplus \mathfrak{v}$ be the Lie algebra of $N_{r,s}$, where $\mathfrak z$ is the centre and $\mathfrak v$ is an admissible module for the Clifford algebra $\Cl(\mathbb R^{r,s})$.
Then the following four assertions hold:

{\rm 1)} $N_{r,s}$ is naturally reductive {\rm(}hence, geodesic orbit{\rm)} if and only if $(r,s)\in\{(0,1), (1,2)\}$;
% and~$\mathfrak{v}$ is any admissible module.

{\rm 2)} $N_{3,4}$ with $\mathfrak v$ being a minimal dimensional admissible Clifford module is geodesic orbit but not naturally reductive;

{\rm 3)} $N_{3,4}$ with $\mathfrak v$ being not minimal dimensional admissible Clifford module is not geodesic orbit;

{\rm 4)} $N_{r,s}$ with $(r,s)\notin \{(0,1), (1,2), (3,4)\}$ is not geodesic orbit for any admissible Clifford module.
\end{theorem}

Theorem~\ref{th.class.pseudo} shows a strong rigidity phenomenon. For instance, in the Riemannian case the GO property is stable under taking isotypic sums of modules for k-dimensional centre with $k\leq 3$ (Theorem~\ref{th.class.riemmannain}). In the pseudo-Riemannian case it holds for all modules only in the two naturally reductive signatures $(0,1)$ and $(1,2)$. The only further GO example is the 15-dimensional group $N_{3,4}(\mathfrak v_{\min})$. This group is the split analogue of the Riemannian case of a seven-dimensional centre, but unlike that case it loses the GO property as soon as $\mathfrak v$ is enlarged, whether by an equivalent or by a non-equivalent minimal module.

The proof combines three ingredients. First, an admissible submodule $\mathfrak v'\subset\mathfrak v$ gives a totally geodesic subgroup $N_{r,s}(\mathfrak v')\subset N_{r,s}(\mathfrak v)$. Since complete totally geodesic submanifolds of GO manifolds are GO (Theorem~\ref{th_ps_totgeod}), every non-GO result for a module passes to all modules containing it. Second, the naturally reductive cases are characterized via Ovando's criterion (Theorem~\ref{th:Ovanto}) together with the structure of the Lie algebra generated by $J(\mathfrak z)$ (Proposition~\ref{pr.triple.1}). Third, for $N_{3,4}$ we carry out an explicit analysis based on a distinguished basis of $\mathfrak v_{\min}$ adapted to the Clifford relations. On $\mathfrak v_{\min}$ this analysis yields all generators of homogeneous geodesics, including the null ones. On $\mathfrak v_{\min}\oplus\mathfrak v_{\min}$ and on $\mathfrak v^-_{\min}\oplus\mathfrak v^+_{\min}$ it yields explicit initial vectors whose geodesics are not homogeneous.

The paper is organized as follows. Section~\ref{se.auiliary} collects preliminaries on pseudo-Riemannian nilmanifolds and pseudo H-type algebras. Section~\ref{sec.example} treats $N_{0,1}$ and $N_{1,2}$ and proves the first statement of Theorem~\ref{th.class.pseudo}. Section~\ref{sec:rs-positive} proves the last statement using totally geodesic subgroups. Section~\ref{sec.N3.4} is devoted to $N_{3,4}$.

\section{Auxiliary results}\label{se.auiliary}

\begin{definition}\label{def:GO-pseudo}
A pseudo-Riemannian homogeneous reductive manifold $(G/H,g)$ is called  geodesic orbit (GO) if
a geodesic through the point $eH$ with any initial vector $\xi$ is an orbit of a one-parameter isometry group of $(G/H,g)$.
\end{definition}

\begin{definition}\label{def:naturaly-reductive}
Let $(M= G/H, g)$ be a homogeneous reductive pseudo-Riemannian
manifold.  Then $M$ is said to be a naturally reductive if there is a reductive decomposition $\mathfrak{g} = \mathfrak{h}\oplus \mathfrak{m}$ such that
\begin{equation} \label{eq:naturaly-reductive}
\langle[T, Q]_{\mathfrak{m}}, R \rangle =\langle [T ,R]_m,Q\rangle
\end{equation}
for the corresponding scalar product and any $T,Q,R\in\mathfrak m$.
\end{definition}

Let $(N,g)$ be a 2-step pseudo-Riemannian nilpotent group with the Lie algebra $\mathfrak{n}$, and  $\langle \cdot\,,\cdot \rangle$ the scalar product  (a symmetric non-degenerate bilinear form) generating the pseudo-Riemannian left invariant metric $g$. We call such groups {\it pseudo-Riemannian nilmanifolds}.
If the centre $\mathfrak{z}$ of $\mathfrak{n}$  is non-degenerate with respect to $\langle \cdot\,,\cdot \rangle$, then we write $\mathfrak{n}=\mathfrak{z}\oplus \mathfrak{v}$,
where
$\mathfrak{v}=\mathfrak{z}^\perp$ relative to $\langle \cdot\,,\cdot \rangle$.
In this case, $\mathfrak{v}$ is also non-degenerate scalar product space, see~\cite{ONeill}.

Whenever $(N,g)$ is connected simply connected, we do not distinguish between the
group of automorphisms of the nilmanifold $(N,g)$ and of its Lie algebra $(\mathfrak{n}=\mathfrak{z} \oplus \mathfrak{v},\langle \cdot\,,\cdot \rangle)$.
Note that each skew-symmetric derivation of $\mathfrak{n}$ leaves each of $\mathfrak{v}$ and $\mathfrak{z}$ invariant.
For any $Z\in \mathfrak{z}$, we consider the operator
\begin{equation}\label{eq_j_z_1}
J_Z:\mathfrak{v} \rightarrow \mathfrak{v}, \mbox{\,\,\,\, such that \,\,\,\,}
\langle J_Z(X),Y \rangle=\langle [X,Y],Z\rangle, \quad X,Y\in \mathfrak{v}.
\end{equation}
The linear map $J_Z$ satisfies $J_Z(X)=(\ad X)'(Z)$, where $(\ad X)'$ is dual to $\ad X(Y)=[X,Y]$ with respect to $\langle \cdot\,, \cdot \rangle$.
The group of isometries $H$ fixing the identity $e\in N$ of the nilmanifold $(N,g)$ is given by

\begin{equation}\label{isotgr1}
H = \{(\varphi,\psi)\in  O(\mathfrak{z},\langle \cdot\,, \cdot \rangle _{\mathfrak{z}})\times O(\mathfrak{v},  \langle \cdot\,, \cdot \rangle _{\mathfrak{v}})
\,|\, \psi J_Z \psi^{-1} = J_{\varphi(Z)},\  Z \in  \mathfrak{z}\},
\end{equation}
while its Lie algebra is
\begin{equation}\label{isotalg1}
\mathfrak{h}= \Der(\mathfrak{n}) \cap \mathfrak{so}(\mathfrak{n},\langle \cdot\,, \cdot \rangle)=
\{D=(C,A)\in  \mathfrak{so}(\mathfrak{z}, \langle \cdot\,, \cdot \rangle _{\mathfrak{z}})\times\mathfrak{so}(\mathfrak{v}, \langle \cdot\,, \cdot \rangle _{\mathfrak{v}})
\,|\, [A,J_Z]= J_{C(Z)},\  Z \in  \mathfrak{z}\}.
\end{equation}

The next result is well known, see e.g. Corollary 3.5 in \cite{CP2009} or Proposition 2.3 in \cite{Ov2013}.

\begin{prop}\label{pr.isgr.1}
Let $(N,g)$ be a pseudo-Riemannian nilmanifold with
non-degenerate center. Then the connected isometry group of $(N,g)$ is
$N\rtimes H$,
where $N$ is the set of left translations by elements of $N$ and the isotropy
subgroup $H$ is given by the isometric automorphisms \eqref{isotgr1} with Lie algebra $\mathfrak{h}$ as in \eqref{isotalg1}.
\end{prop}

In this case,
the isotropy group of $(N,g)$ at the identity element $e$ is exactly $H$ with the embedding $a \in H \mapsto (e,a)\in N\rtimes H$.

\begin{prop}[Geodesic Lemma \cite{DK2007}]\label{pr.geodlem1}
Let $(M= G/H, g)$ be a homogeneous reductive pseudo-Riemannian
manifold, with the corresponding reductive decomposition $\mathfrak{g} = \mathfrak{h}\oplus \mathfrak{m}$. Then $M$
is a $G$-geodesic orbit space if and only if, for any $T\in  \mathfrak{m}$, there exist $P = P(T) \in \mathfrak{h}$
and $k = k(T) \in \mathbb{R}$ such that
\begin{equation} \label{geodlem1}
\langle[T + P, Q]_{\mathfrak{m}}, T \rangle = k \,\langle T , Q \rangle\quad\text{ for any $Q \in \mathfrak{m}$},
\end{equation}
where  $\langle \cdot\, ,\cdot \rangle$ denotes the inner product on  $\mathfrak{m}$ defined by the pseudo-Riemannian metric $g$, and the subscript $\mathfrak{m}$ in \eqref{geodlem1}
means taking the $\mathfrak{m}$-component in $\mathfrak{g} = \mathfrak{h}\oplus \mathfrak{m}$.
\end{prop}

\begin{remark}
Note that $k(T) = 0$ unless $T$ is a null vector.  It suffices to substitute $Q = T$ in \eqref{geodlem1}.
\end{remark}

Proposition~\ref{pr.geodlem1} for $2$-step nilpotent pseudo-Riemannian groups can be rewritten in the spirit of work~\cite{Gor96},
where this idea was used for Riemannian metrics on nilpotent Lie groups, see also~\cite{Bar, GorNik2018}.

\begin{prop}[Section 2 \cite{FuMaNik25}]\label{gonil1n.2}
A pseudo-Riemannian nilmanifold  $(N,g)$ is geodesic orbit  if and only if for any $Z\in \mathfrak{z}$ and
$X\in \mathfrak{v}$ there is $D=(C,A)\in\mathfrak h$ such that
$$
\Big(C+k\Id \Big)(Z)=0,\quad \Big(A+k\Id\Big)(X)=J_Z(X).
$$
\end{prop}

\begin{remark}
The defining equality~\eqref{eq_j_z_1} for $J_Z$ shows that $J_Z\in\mathfrak{so}(\mathfrak v,\langle \cdot\,, \cdot \rangle _{\mathfrak{v}}))$ for any $Z\in\mathfrak z$.
We set $\mathbf{V}=J(\mathfrak{z})$ to be a linear subspace in $\mathfrak{so}(\mathfrak{v},\langle \cdot\,, \cdot \rangle _{\mathfrak{v}})$.
Further, for every $D=(C,A)\in \mathfrak{h}$ and $Z\in \mathfrak{z}$ we get
$J_{C(Z)}=[A,J_Z]$ by the condition on $D$ to be a skew-symmetric derivation. Then every element of $\mathfrak{h}$
is determined by a skew-symmetric operator $A \in \mathfrak{so}(\mathfrak{v}, \langle \cdot\,, \cdot \rangle _{\mathfrak{v}})$ with the property
$[A,\mathbf{V}]\subset \mathbf{V}$ $\big($i.e., $A$ is from the normalizer of $\mathbf{V}$ in $\mathfrak{so}(\mathfrak{v}, \langle \cdot\,, \cdot \rangle _{\mathfrak{v}})\big)$.
In this case the operator $C\in \mathfrak{so}(\mathfrak{z}, \langle \cdot\,, \cdot \rangle _{\mathfrak{z}})$ can be recovered from the equality
$J_{C(Z)}=[A,J_Z]$ for all $Z\in  \mathfrak{z}$.
\end{remark}

A special case of geodesic orbit pseudo-Riemannian spaces are naturally reductive homogeneous pseudo-Riemannian spaces, see Definition~\ref{def:naturaly-reductive}, and also~\cite{Ov2013} with references therein.
The naturally reductive Riemannian homogeneous spaces are generalizations of normal homogeneous Riemannian spaces
and symmetric spaces, see, e.g., \cite{AFF,DZ,Gor85,KV85,Storm19,Storm20}.

\begin{theorem}[Theorem 3.2~\cite{Ov2013}]\label{th:Ovanto}
Let $(N,g)$ denote a 2-step pseudo-Riemannian nilmanifold with a non-degenerate center. Assume that the map $J\colon \mathfrak{z}\to \mathfrak{so}(\mathfrak{v})$ is injective, see~\eqref{eq_j_z_1}.
Then the metric is naturally reductive with respect to
$G= N\rtimes H$ {\rm(}see Proposition~\ref{pr.isgr.1}{\rm)}, if and only if
\begin{itemize}
\item[(i)] $\mathbf {V}=J(\mathfrak{z})$ is a Lie subalgebra of $\mathfrak{so}(\mathfrak{v})$;
\item[(ii)] $[J(Z_1), J(Z_2)] = J\big(\tau_{Z_1}(Z_2)\big)$ where $\tau_{Z_1}\in\mathfrak{so}(\mathfrak{z})$ for any $Z_1\in \mathfrak{z}$.
\end{itemize}
\end{theorem}
Since the map $J$ is supposed to be injective, the map $\tau$ can be easily recovered from ${\rm(ii)}$.
%%%%%%%%%%%%%%%%%%%%%%%%%%%%%%%%%%%%%%%%%%%%%

\subsection{Pseudo-Riemannian $H$-type nilmanifolds}\label{se.rseudo.htnlg}

Let $(N,g)$ be a 2-step pseudo-Riemannian nilmanifold and $\mathfrak n=\mathfrak z\oplus\mathfrak v$ its Lie algebra endowed with a scalar product $\langle \cdot\,, \cdot \rangle$
making the center non-degenerate. We identify $(\mathfrak z,\langle \cdot\,, \cdot \rangle_{\mathfrak z})$ with the pseudo Euclidean vector space
$\mathbb{R}^{r,s}=(\mathbb{R}^{r+s},\langle \cdot\,, \cdot \rangle_{r,s})$, where
$$
\langle Z, W \rangle_{r,s} =
\sum_{i=1}^r z_i\, w_i-\sum_{j=1}^s z_{r+j}\, w_{r+j}, \quad Z,W \in \mathbb{R}^{r+s}.
$$
If the linear operator $J\colon\mathfrak{z} \rightarrow \mathfrak{so}(\mathfrak{v},\langle \cdot\,, \cdot \rangle_{\mathfrak v})$ defined by
\begin{equation}\label{eq:JJ}
\langle J_Z(X),Y \rangle_{\mathfrak v}=\langle [X,Y],Z\rangle_{r,s}, \quad X,Y\in \mathfrak{v},\ Z\in \mathfrak{z}=\mathbb{R}^{r,s},
\end{equation}
satisfies $J_Z^2 (X)=-\langle Z,Z \rangle_{r,s} X$ for any $Z \in \mathbb{R}^{r,s}$ and all $X \in \mathfrak{v}$,
then $\mathfrak{n}=\mathfrak n_{r,s}$ is called the pseudo $H$(eisenberg)-type Lie algebra. It is easy to check that
this definition %the definition of {\color{blue}pseudo} $H$-type Lie algebra
implies
\begin{equation}\label{eq:JJJ}
\langle J_Z(X),J_W(X) \rangle_{\mathfrak v}=\langle Z,W \rangle_{r,s}\langle X,X \rangle_{\mathfrak v}.
\end{equation}
We denote by ${N}_{r,s}$ the connected simply connected Lie group, whose Lie algebra is the pseudo $H$-type Lie algebra $\mathfrak n_{r,s}$. The $H$-type Lie groups $N_{r,0}$ with a positive definite scalar product were introduced in~\cite{Kap80}
and with an arbitrary indefinite scalar product in~\cite{Cia}, see also~\cite{GKM13}.

These
Lie algebras are related to the representations of the
Clifford algebras in the following way. Let $J\colon\operatorname{Cl}(\mathbb{R}^{r,s}) \rightarrow \operatorname{End} (\mathfrak{v})$ be a representation of
the Clifford algebra $\operatorname{Cl}(\mathbb{R}^{r,s})$ generated by the pseudo Euclidean vector space $\mathbb{R}^{r,s}$. If the representation can be reduced to $J\colon\operatorname{Cl}(\mathbb{R}^{r,s}) \rightarrow \mathfrak{so}(\mathfrak{v},\langle \cdot\,,\cdot \rangle_{\mathfrak{v}})$, then we get a pseudo $H$-type Lie algebra $\mathfrak n_{r,s}=(\mathfrak z\oplus\mathfrak v, [\cdot\,,\cdot],\langle \cdot\,,\cdot \rangle_{\mathfrak{z}}+\langle \cdot\,,\cdot \rangle_{\mathfrak{v}})$ with the Lie bracket $[\cdot\,,\cdot]$ defined in~\eqref{eq:JJ},
see details in~\cite{Cia,FuMa17, FuMa19, FuMa21,GMKMV, Kap81,  KapTir}. The scalar product $\langle \cdot\,,\cdot \rangle_{\mathfrak{v}}$ in this case is called admissible and $\mathfrak v$ is called admissible Clifford module. Thus we obtain
$$
\langle J_Z(X),Y \rangle_{\mathfrak v}=-\langle X,J_Z(Y) \rangle_{\mathfrak v},\quad \text{for any }\quad Z\in \mathbb{R}^{r,s},\quad X,Y\in\mathfrak v.
$$

The space $(\mathfrak{v},\langle \cdot\,,\cdot \rangle_{\mathfrak{v}})$ is isometric to $\mathbb{R}^{l,l}$ for some $l \in \mathbb{N}$ if $s>0$,
whereas $(\mathfrak{v},\langle \cdot\,,\cdot \rangle_{\mathfrak{v}})$ is isometric either to $\mathbb{R}^{2l,0}$ or $\mathbb{R}^{0,2l}$  if $s=0$, see e.g.~\cite{Cia}, or~\cite[Proposition 2.2.2]{FuMa21}).
We use the identification $\mathbf{V}=J(\mathfrak{z}) \subset \mathfrak{so}(\mathfrak{v},\langle \cdot\,,\cdot \rangle_{\mathfrak{v}})\cong\mathfrak{so}(l,l)$, $\dim \mathbf{V}=\dim \mathfrak{z}=r+s$. The groups of automorphisms and isometries for $H$-type nilmanifolds with a positive definite metric were studied in~\cite{Rie82,Saal96} and for $H$-type nilmanifolds with indefinite metric in~\cite{FuMa21}.

\section{Pseudo-Riemannian $H$-type nilmanifolds $N_{0,1}$ and $N_{1,2}$}\label{sec.example}

\subsection{Pseudo-Riemannian $H$-type nilmanifold $N_{0,1}$}\label{sec.N01}
We denote by $N_{0,1}$ the $H$-type nilmanifold with the Lie algebra isometric to
$\mathbb{R}^{0,1}\oplus\mathbb{R}^{l,l}$. The Clifford module $(\mathfrak v, \langle \cdot\,, \cdot \rangle_{l,l})\cong \mathbb{R}^{l,l}$ can be written as a direct sum $\mathfrak v=\oplus_{j=1}^l(\mathfrak v_{\min})_j$ of $l$ minimal dimensional admissible modules $\mathfrak v_{\min}\cong\mathbb R^{1,1}$.
We choose an orthonormal basis $\{X_1,\ldots X_l,Y_1\ldots,Y_l,Z\}$ satisfying
$$
\langle X_i,X_i\rangle_{l,l}=-\langle Y_i,Y_i\rangle_{l,l}=1,\quad \langle Z,Z\rangle_{0,1}=-1,\quad [X_i,Y_i]=Z,
$$
and all other commutators vanish. Then
$$
J_Z=\begin{pmatrix}0&\Id_l\\\Id_l&0\end{pmatrix}\in \mathfrak{so}(l,l),\quad J_Z^2=-\langle Z,Z\rangle_{0,1} \Id_{2l}=\Id_{2l}.
$$
It is clear that $J_Z$
spans a one dimensional subspace $\mathbf{V}=J(\mathfrak{z})\subset \mathfrak{so}(l,l)$. The space $\mathbf{V}$ is an abelian subalgebra of
$\mathfrak{so}(l,l)$ and
$$
[J(Z_1), J(Z_2)] = [J(aZ), J(bZ)] =ab[J(Z), J(Z)] =0=J\big(\tau_{Z_1}(Z_2)\big),\quad a,b\in\mathbb{R},
$$
for the operator $\mathfrak{so}(\mathfrak{z},\langle \cdot\,,\cdot\rangle_{\mathfrak z})\ni\tau_{Z_1}\equiv 0$ for any $Z_1\in\mathfrak{z}$.
Therefore, we get naturally reductive hence, geodesic orbit nilmanifold by Theorem~\ref{th:Ovanto}.

\subsection{Pseudo-Riemannian $H$-type nilmanifold $N_{1,2}$}\label{sec.N12}

The $H$-type nilmanifold $N_{1,2}$ is the $H$-type Lie group of dimension $4l+3$, $l\geq 1$, with the Lie algebra isometric to $\mathbb{R}^{1,2}\oplus\mathbb{R}^{2l,2l}$.  In this  case we need to explain the structure of an admissible Clifford module in details, see~\cite{FuMa19}.
The Clifford algebra $\Cl(\mathbb R^{1,2})$ has two irreducible non-equivalent modules $V_+,V_-$ of dimension $\dim(V_+)=\dim(V_-)=2$. Moreover
$$
J_{Z_1}J_{Z_2}J_{Z_3}x=x\quad\text{for}\quad x\in V_+,\qquad
J_{Z_1}J_{Z_2}J_{Z_3}x=-x\quad\text{for}\quad x\in V_-.
$$
To define a minimal dimensional admissible module, which will carry a non-degenerate scalar product one needs to consider the vector space $\mathfrak v_{\min}=V_+\times V_-$ with a scalar product such that
$$
\langle x,x\rangle_{\mathfrak v_{\min}}>0\quad\text{for}\quad x\neq 0,\ x\in V_+,\qquad
\langle x,x\rangle_{\mathfrak v_{\min}}<0\quad\text{for}\quad x\neq 0,\  x\in V_-.
$$
Thus a minimal admissible module $\mathfrak v_{\min}$ is isometric to $\mathbb R^{2,2}$ and it is a direct sum of two non-equivalent Clifford modules, endowed by a neutral metric. For the minimal admissible module $\mathfrak v_{\min}$ one can construct an orthonormal basis by the following. Let $v\in\mathfrak v_{\min}$ be such that $J_{Z_1}J_{Z_2}J_{Z_3}v=v$ and $\langle v,v\rangle_{\mathfrak v_{\min}}=1$. Then we define the orthonormal basis by
\begin{equation}\label{eq:basis12}
X_1=v,\quad X_2=J_{Z_1}v,\quad X_3=J_{Z_2}v,\quad X_4=J_{Z_3}v=J_{Z_1}J_{Z_2}v.
\end{equation}
Note that now $J_{Z_1}J_{Z_2}J_{Z_3}x=x$ for any $x\in \mathfrak v_{\min}$ due to the construction of admissible module and~\eqref{eq:JJJ}.

Any admissible module of a Clifford algebra is completely reducible to a sum of minimal admissible modules, see~\cite[Section 4.1]{FuMa19}. Hence an admissible Clifford module for $\Cl(\mathbb R^{1,2})$ can be written as an orthogonal direct sum $\mathfrak v=\oplus_{j=1}^l(\mathfrak{v}_{\min})_j$ of $l$ minimal admissible modules isometric to $\mathbb R^{2,2}$. The operators $J_{Z_i}$ act block-diagonally with the same algebraic relations on each summand.
We choose $v_j\in (\mathfrak v_{\min})_j$ such that $J_{Z_1}J_{Z_2}J_{Z_3}v_j=v_j$, $\langle v_j,v_j\rangle_{\mathfrak v_{\min}}=1$, and define the orthonormal basis $\{X_{j1},X_{j2},X_{j3},X_{j4},Z_1,Z_2,Z_3\}_{j=1}^l$ for $\mathfrak v\oplus\mathfrak z=\oplus_{j=1}^l(\mathfrak{v}_{\min})_j\oplus\mathfrak z$ as in~\eqref{eq:basis12}. Then the basis satisfies
$$
\langle X_{j1},X_{j1}\rangle_{2l,2l}=\langle X_{j2},X_{j2}\rangle_{2l,2l}=
-\langle X_{j3},X_{j3}\rangle_{2l,2l}=-\langle X_{j4},X_{j4}\rangle_{2l,2l}=1,\qquad j=1,\ldots,l,
$$
$$
\langle Z_1,Z_1\rangle_{1,2}=-\langle Z_2,Z_2\rangle_{1,2}=-\langle Z_3,Z_3\rangle_{1,2}=1,
$$
and has non-vanishing commutation relations
$$
[X_{j1},X_{j2}]=-[X_{j3},X_{j4}]=Z_1,\quad [X_{j1},X_{j3}]=-[X_{j2},X_{j4}]=Z_2,\quad [X_{j1},X_{j4}]=[X_{j2},X_{j3}]=Z_3.
$$
Note also that under this construction all the minimal admissible modules are equivalent because of $J_{Z_1}J_{Z_2}J_{Z_3}x=x$ for any $x\in \mathfrak v_{\min}$, see also~\cite[Theorem 4.1.1]{FuMa19}.

The commutation relations
$$
2J_{Z_3}=2J_{Z_1}J_{Z_2}=[J_{Z_1},J_{Z_2}], \quad -2J_{Z_2}= 2J_{Z_1}J_{Z_3}= [J_{Z_1},J_{Z_3}],\quad
-2J_{Z_1}=2J_{Z_2}J_{Z_3}=[J_{Z_2},J_{Z_3}],
$$
show that $\mathbf {V}=\operatorname{span}\{J_{Z_1},J_{Z_2},J_{Z_3}\}\subset \mathfrak{so}(2,2)$ is a Lie subalgebra.
We conclude that pseudo-Riemannian $H$-type nilmanifold $N_{1,2}$ is naturally reductive by Theorem~\ref{th:Ovanto}, hence, geodesic orbit.

\subsection{Proof of the first statement in Theorem~\ref{th.class.pseudo}}

Sections~\ref{sec.N01} and~\ref{sec.N12} allow us to obtain a complete classification of naturally reductive pseudo-Riemannian $H$-type nilmanifolds~$N_{r,s}$.
\begin{prop}[Proposition 4, Lemma 1 ~\cite{FuMaNik25}]\label{pr.triple.1}
Let $\mathfrak{n}_{r,s}$ be a pseudo $H$-type Lie algebra, $\mathbf{N}$ and $\mathbf{Z}$ the normalizer and the centralizer of $\mathbf{V}=J(\mathfrak z)$ in $\mathfrak{so}(l,l)$. Then we have the following properties:

1. $[\mathbf{V},\mathbf{V}]$ and $\mathbf{L}:=[\mathbf{V},\mathbf{V}]+\mathbf{V}$ are Lie subalgebras in $\mathfrak{so}(l,l)$;
\smallskip

2. The Lie algebra $[\mathbf{V},\mathbf{V}]$ is isomorphic to $\mathfrak{so}(r,s)$;
\smallskip

3. The Lie algebra $\mathbf{L}~\text{is commutative if}~
(r, s)\in\left\{(1,0), (0,1)\right\}$,
it is simple if\\ %\linebreak
\hspace{0.8cm} $(r, s)\not\in \left\{(1,0), (0,1),(3,0), (1,2)\right\}$,
and it is semisimple if $(r, s)\in \left\{(3,0), (1,2)\right\}$;
\smallskip

4. If $\mathbf{V}\cap [\mathbf{V},\mathbf{V}]$ is nontrivial and $\mathbf{L}$ is simple, then $\mathbf{V}=[\mathbf{V},\mathbf{V}]=\mathbf{L}$;
\smallskip

5.  $\mathbf{N}=[\mathbf{V},\mathbf{V}] \oplus \mathbf{Z}$ {\rm(}a direct sum of Lie algebras{\rm)}.
\end{prop}

\begin{theorem}\label{pr.natred.1}
A pseudo-Riemannian $H$-type nilmanifold $N_{r,s}$ is naturally reductive if and only if \linebreak $(r,s)\in\{(1,0),(0,1),(3,0),(1,2)\}$.
\end{theorem}

\begin{proof}
We know that $N_{1,0}$, $N_{0,1}$, $N_{3,0}$, and $N_{1,2}$  are naturally reductive, see Theorem~\ref{th.class.riemmannain} and Sections~\ref{sec.N01},~\ref{sec.N12}.

If $r+s =1$, then a pseudo-Riemannian $H$-type nilmanifold is isometric to $N_{1,0}$ or to $N_{0,1}$ (independently of being naturally reductive or not).

Suppose that $N_{r,s}$ is naturally reductive and $r+s >1$.
Then the linear space $\mathbf{V}=J(\mathfrak{z})$ is a Lie subalgebra in
$\mathbf{L}= \mathbf{V}+[\mathbf{V},\mathbf{V}]$. Since $[\mathbf{V},[\mathbf{V},\mathbf{V}]] \subset \mathbf{V}$,
then $\mathbf{V}$ is an ideal in $\mathbf{L}$.

If $(r,s)\not\in \{(3,0),(1,2)\}$, then $\mathbf{L}$ is simple and $\mathbf{L}=[\mathbf{V},\mathbf{V}]=\mathbf{V}$ by Proposition~\ref{pr.triple.1}.
Since $\dim(\mathbf{V}) = r+s$ and $\dim([\mathbf{V},\mathbf{V}]) = (r+s)(r+s-1)/2$, we obtain $r+s = 3$ which implies
$(r,s) \in \{(2,1),(0,3)\}$.
The final observation is the following: it was shown in~\cite[Sections 4.5, 4.6]{FuMaNik25} that
$N_{2,1}(\mathfrak v_{\min})$ and $N_{0,3}(\mathfrak v_{\min})$, constructed from minimal admissible modules $\mathfrak v_{\min}$, are not geodesic orbit.
Proposition~\ref{co_totg_2} from Section~\ref{se.tot.geod} extends this result to $N_{2,1}(\mathfrak v)$ and $N_{0,3}(\mathfrak v)$,
where $\mathfrak v$ is an arbitrary admissible module.
We therefore conclude that $N_{2,1}(\mathfrak v)$ and $N_{0,3}(\mathfrak v)$ are not naturally reductive.
\end{proof}

\section{Pseudo-Riemannian $H$-type nilmanifolds $N_{r,s}$, $s>0$, $(r,s)\notin \{(0,1),(1,2),(3,4)\}$}\label{sec:rs-positive}

\subsection{On totally geodesic submanifolds of geodesic orbit pseudo-Riemannian manifolds}\label{se.tot.geod}

There is a close relation between the geodesic orbit submanifolds and totally geodesic submanifolds.
P.~Eberlein~\cite{Eber94.1, Eber94.2} studied totally geodesic subalgebras and totally
geodesic subgroups of non-singular 2-step nilpotent Lie groups endowed with left-invariant Riemannian metrics.
The pseudo $H$-type Lie algebras are examples of non-singular 2-step nilpotent Lie algebras.
It was proved in~\cite[Theorem 11]{BerNik08} that
every closed totally geodesic submanifold of a
GO Riemannian manifold is
GO itself. A version of these results for pseudo-Riemannian manifolds was proved in~\cite[Section 5]{FuMaNik25}. We state it here for the completeness.

\begin{theorem}[Theorem 5~\cite{FuMaNik25}]\label{th_ps_totgeod}
Every geodesically complete totally geodesic submanifold of a
geodesic orbit pseudo-Riemannian manifold is
geodesic orbit itself.
\end{theorem}
Note that any GO pseudo-Riemannian manifold is geodesically complete.
Theorem~\ref{th_ps_totgeod} will be useful for us in the following reformulation.

\begin{prop}\label{co_totg_2}
Let $\big(N_{r,s}(U),g_{N_{r,s}(U)}\big)$ be a pseudo-Riemannian $H$-type nilmanifold, where $U$ is an admissible module of the representation $J\colon\Cl(\mathbb R^{r,s}) \to \operatorname{End} (U)$. Let $V$ be an admissible submodule of $U$ with respect to the restriction $g_{N_{r,s}(V)}$ of the metric. If $\big(N_{r,s}(U),g_{N_{r,s}(U)}\big)$ is geodesic orbit, then
a pseudo-Riemannian $H$-type nilmanifold $\big(N_{r,s}(V),g_{N_{r,s}(V)}\big)$ is also geodesic orbit.
\end{prop}

\begin{proof}
The manifold $\big(N_{r,s}(V),g_{N_{r,s}(V)}\big)$  is a totally geodesic submanifold of $\big(N_{r,s}(U),g_{N_{r,s}(U)}\big)$ by~\cite[Theorem~5]{FuMaNik25}.
Applying Theorem~\ref{th_ps_totgeod}, we finish the proof.
\end{proof}

\subsection{Proof of the last statement in Theorem~\ref{th.class.pseudo}}

We recall the main result obtained in~\cite{FuMaNik25}.

\begin{theorem}\label{th.class.pseudo-0}
Let $N_{r,s}(\mathfrak v_{\min})$ be an $H$-type group, where $(r,s)$, $s\geq 1$, is the signature of the left invariant pseudo-Riemannian metric restricted to the centre of the group. Let $\mathfrak{n}_{r,s}=\mathfrak{z}\oplus \mathfrak{v}_{\min}$ be the Lie algebra of $N_{r,s}$, where $\mathfrak z$ is the centre and $\mathfrak v_{\min}$ is a minimal admissible module for the Clifford algebra $\Cl(\mathbb R^{r,s})$.
Then $N_{r,s}(\mathfrak v_{\min})$ with $(r,s)\notin \{(0,1), (1,2), (3,4)\}$ is not geodesic orbit pseudo-Riemannian manifold.
\end{theorem}
{\bf Proof of the last statement of Theorem~\ref{th.class.pseudo}.}
Let $N_{r,s}\big(\oplus_{j=1}^l(\mathfrak{v}_{\min})_j\big)$ be a pseudo-Riemannian $H$-type nilmanifold with $(r,s)\notin \{(0,1), (1,2), (3,4)\}$. Then $N_{r,s}(\mathfrak v_{\min})$ is a totally geodesic submanifold of $N_{r,s}\big(\oplus_{j=1}^l(\mathfrak{v}_{\min})_j\big)$ by  Proposition~\ref{co_totg_2}.
Theorem~\ref{th.class.pseudo-0} implies that $N_{r,s}\big(\oplus_{j=1}^l(\mathfrak{v}_{\min})_j\big)$ can not be geodesic orbit manifold since $N_{r,s}(\mathfrak v_{\min})$ is not geodesic orbit. This proves the last statement of Theorem~\ref{th.class.pseudo}.
\ {\tiny $\blacksquare$}

\section{Pseudo-Riemannian $H$-type nilmanifold ${N}_{3,4}$}\label{sec.N3.4}

In the present section we finish the proof of Theorem~\ref{th.class.pseudo}.

\subsection{Pseudo-Riemannian $H$-type nilmanifold ${N}_{3,4}(\mathfrak v_{\min})$}\label{sec:34min}

It is known that the pseudo-Riemannian $H$-type nilmanifold ${N}_{3,4}(\mathfrak v_{\min})$,
where the complement to the centre is the minimal admissible module is a geodesic orbit manifold, see~\cite[Theorem~14]{FuMaNik25}.
We present here an alternative proof of the following theorem, which gives an additional information about the dimension of free parameters
in the choice of homogeneous geodesics.

\begin{theorem}\label{th:34min}
The space $N_{3,4}(\mathfrak v_{\min})$ is a geodesic orbit manifold.
Each geodesic is the orbit of a one-parameter isometry group. For every initial
vector the generators of such groups form a non-empty affine subspace, which is
eight-dimensional for a generic initial vector.
\end{theorem}

It is known that the Clifford algebra $\operatorname{Cl}(\mathbb{R}^{3,4})$ has two non-equivalent minimal admissible $8$-dimensional modules, which in this case are irreducible Clifford modules, see~\cite{ABS64,LowMich89}.
Moreover, the corresponding $15$-dimensional pseudo $H$-type Lie algebras are isomorphic and isometric, see~\cite[Theorem 12]{FuMa17}.
Therefore, it is sufficient to check only one such pseudo $H$-type nilmanifold $N_{3,4}(\mathfrak v_{\min})$.
The pseudo-Riemannian $H$-type nilmanifold $N_{3,4}$ carries a left invariant metric generated by
the scalar product
$\langle \cdot\, , \cdot\rangle_{3,4}+\langle \cdot\, , \cdot\rangle_{4,4}$
on the Lie algebra $\mathfrak{n}_{3,4}=\mathbb R^{3,4}\oplus\mathbb R^{4,4}$.
We write $(z_0,x_0)\in \mathfrak z\oplus\mathfrak v_{\min}\cong\mathbb R^{3,4}\oplus \mathbb R^{4,4}$ for an initial vector for the geodesic. For different types of the initial vectors we will use distinguished bases for the horizontal space $\mathfrak v_{\min}$ and for the centre $\mathfrak z$. We start from a basis for the centre.

\subsubsection{Central component $z_0$ of initial vectors}\label{central z_0}

We are going to use the fact that there are only three classes in the centre of $\mathfrak{n}_{3,4}$, up to similarity and the action of the isotropy subgroup.
They are represented by $Z_1$ (positive vector), $Z_4$ (negative vector), and $Z_1+Z_4$ (null vector).
Indeed, the group $O(3,4)$ with the Lie algebra $\mathfrak{so}(3,4)$, that is a linear span of the operators $[J_{Z_i},J_{Z_j}]$, $1\leq i <j \leq 7$,
acts naturally by automorphisms on the centre $\mathfrak{z}$. It is well known that $O(3,4)$ acts transitively on every hyperquadric
$Q(\rho)=\{z \in \mathfrak{z} \,|\,\langle z, z\rangle_{3,4}=\rho\}$ with $\rho\neq 0$, see e.g.~\cite[page~239]{ONeill}.

Let us suppose that $z,\overline{z} \in \mathfrak{z}$ and $\langle z, z\rangle_{3,4}=\langle \overline{z}, \overline{z} \rangle_{3,4}=0$. If
$\langle z_{\mathbf p}, z_{\mathbf p}\rangle_{3,4}=\langle \overline{z}_{\mathbf p}, \overline{z}_{\mathbf p} \rangle_{3,4}=\rho\neq 0$ or, equivalently,
$\langle z_{\mathbf n}, z_{\mathbf n}\rangle_{3,4}=\langle \overline{z}_{\mathbf n}, \overline{z}_{\mathbf n} \rangle_{3,4}=-\rho\neq 0$, where the subscripts ${\mathbf p}$ and ${\mathbf n}$ mean the components
of vectors in $\spann\{Z_1,Z_2,Z_3\}$ and $\spann\{Z_4,Z_5,Z_6,Z_7\}$, respectively.
Then, it is easy to see, that there is $Q\in O(3)\times O(4) \subset O(3,4)$ such that $Q(z_{\mathbf p})=\overline{z}_{\mathbf p}$ and $Q(z_{\mathbf n})=\overline{z}_{\mathbf n}$.
Hence, up to a similarity, it suffices to consider only one point in every of the following hyperquadrics: $Q(1)$, $Q(-1)$, $Q(0)$ ( a nonzero point in the last case).

We can fix the part $z_0\in\mathfrak z$ of an initial vector $(z_0,x_0)$ of a geodesic in three different ways.

{\bf (C1)} If $\|z_0\|^2=\langle z_0, z_0\rangle_{\mathbb R^{3,4}}=1$, then we set
\[
\{Z_1=z_0,Z_2,\ldots, Z_7\},\quad
\]
as an orthonormal basis, satisfying
\begin{equation}\label{eq:b7}
\|Z_1\|^2=\|Z_2\|^2=\|Z_3\|^2=1,\quad \|Z_4\|^2=\|Z_5\|^2=\|Z_6\|^2=\|Z_7\|^2=-1.
\end{equation}
\smallskip

{\bf (C2)} If $\|z_0\|^2=\langle z_0, z_0\rangle_{\mathbb R^{3,4}}=-1$, then we set
\[
\{Z_1,Z_2,Z_3,Z_4=z_0,Z_5,Z_6,\ldots, Z_7\},\quad
\]
such that~\eqref{eq:b7} holds.
\smallskip

{\bf (C3)} If $\|z_0\|^2=\langle z_0, z_0\rangle_{\mathbb R^{3,4}}=0$, then first we fix an arbitrary orthonormal basis
 $\{z_1^{*},\ldots,z_7^{*}\}$, satisfying ~\eqref{eq:b7} and write
 \[
 z_0=\sum_{i=1}^3~\theta_i\,z_i^{*}+\sum_{i=4}^7\theta_i\,z_i^{*}:=z_+^{*} + z_-^{*}.
 \]
Then, without loss of generality, we can assume $\sum_{i=1}^3 {\theta_i}^2=\sum_{i=4}^{7}{\theta_i}^2=1$.
 Since $\langle z_+^{*},~z_-^{*}\rangle_{\mathbb R^{3,4}}=0$ we can
 select an orthonormal basis
\[
\{Z_1=z_+^{*},Z_2,Z_3,Z_4=z_-^{*},Z_5,Z_6, Z_7\},\quad
\]
still satisfying~\eqref{eq:b7}. With these basis vectors we can denote $z_0=Z_1+Z_4$.

In what follows we use the construction of basis on the centre as above.

\begin{remark}
In any of the above constructions, the operators $J_{Z_j}$, $j=1,\dots,7$, span a $7$-dimensional subspace in $J({\mathfrak z})\subset\mathfrak{so}(4,4)$, nevertheless the vector space $\mathbf{V}=J({\mathfrak z})$ is not a Lie subalgebra of $\mathfrak{so}(4,4)$. Hence, $N_{3,4}(\mathfrak v_{\min})$ is not a naturally reductive manifold, see Theorem~\ref{th:Ovanto}.
\end{remark}

\subsubsection{Horizontal component $x_0$ of initial vectors and reduction by conjugate Clifford action}

Fix $z_0\in\mathfrak z$ and choose a suitable basis for the centre as was described in Section~\ref{central z_0}. Now we have three cases.
\smallskip
$$
{\bf (H1)} \ \|x_0\|^2=\langle x_0, x_0\rangle_{\mathbb R^{4,4}}=1,\quad
{\bf (H2)} \ \|x_0\|^2=-1,\quad
{\bf (H3)} \ \|x_0\|^2=0.
$$

Before we consider these three options we let $v\in\mathfrak v_{\min}$ be such that $\|v\|^2=1$ and
\begin{equation}\label{eq:inv34}
\begin{array}{lllll}
&J_{Z_1}J_{Z_2}J_{Z_4}J_{Z_5}(v)=v,\quad &J_{Z_1}J_{Z_2}J_{Z_6}J_{Z_7}(v)=v,\quad
\\
&J_{Z_1}J_{Z_3}J_{Z_5}J_{Z_7}(v)=v,\quad
&J_{Z_1}J_{Z_2}J_{Z_3}(v)=v.
\end{array}
\end{equation}
The condition $J_{Z_1}J_{Z_2}J_{Z_3}(v)=v$ is equivalent to $J_{Z_1}J_{Z_2}J_{Z_3}J_{Z_4}J_{Z_5}J_{Z_6}J_{Z_7}(v)=-v$. This choice fixes the module $\mathfrak v_{\min}=\mathfrak v^{-}_{\min}$. Moreover, the choice of $v$ as in~\eqref{eq:inv34} allows us to use the following orthonormal basis for the module $\mathfrak v_{\min}$:
\begin{equation}\label{eq:basis34-7}
\begin{array}{lllll}
&X_1=v,\quad &X_2=J_{Z_1}(v),\quad &X_3=J_{Z_2}(v),\quad &X_4=J_{Z_3}(v),
\\
&X_5=J_{Z_4}(v),\quad &X_6=J_{Z_5}(v),\quad &X_7=J_{Z_6}(v),\quad &X_8=J_{Z_7}(v).
\end{array}
\end{equation}
\smallskip

Let $x_0\in\mathfrak v_{\min}$ be as in $\bf (H1)$ above. Let $T\in SO(4,4)$ be such that $x_0=T(v)$ and
consider the conjugate Clifford action
\begin{equation}\label{eq:tildeT}
\tilde{J}_w:=T^{-1}J_wT=T^{\tau}J_{w}T,
\end{equation}
where $ T^{\tau}$ is the transpose to $T$ with respect to the scalar product $\langle \cdot\,,\cdot \rangle_{4,4}$.
Then, a geodesic determined by the initial vector $(z_0,x_0)$
is homogeneous if there is a pair $D=(C,A)$ such that
\begin{align}\label{in nor + +}
&\Big(A+k\Id\Big)(x_0)=J_{z_0}(x_0),\quad \Big(C+k\Id\Big)(z_0)=0,\quad [A,J_w]=J_{C(w)},\quad\text{for all}\ w \in\mathbb{R}^{3,4}.
\end{align}
\begin{lemma}\label{lem:1}
Let $x_0=T(v)$ for $T\in SO(4,4)$. Then condition~\eqref{in nor + +} is equivalent to
\begin{equation}\label{in* nor* + +}
\Big(\tilde{A}+k\Id\Big)(v)=\tilde{J}_{z_0}(v),\quad  \Big(C+k\Id\Big)(z_0)=0, \quad [\tilde{A}, \tilde{J}_w]=\tilde{J}_{C(w)},\quad w\in\mathfrak z\cong\mathbb{R}^{3,4}
\end{equation}
where $\tilde{A}(v)=T^{-1}AT(v)$ and $\tilde{J}_w$ as in~\eqref{eq:tildeT}.
\end{lemma}
\begin{proof}
In fact, from~\eqref{eq:tildeT} and~\eqref{in nor + +}
\begin{eqnarray*}
\Big(\tilde{A}+k\Id\Big)(v)&=&T^{-1}\Big(A+k\Id\Big)T(v)=T^{-1}\Big(A+k\Id\Big)(x_0)=T^{-1}J_{z_0}(x_0)
\\
&=&T^{-1}T\tilde J_{z_0}T^{-1}(x_0)
= \tilde{J}_{z_0}(v),
\end{eqnarray*}
$$
[\tilde{A}, \tilde{J}_w]=[T^{-1}AT, T^{-1}J_w T]=T^{-1}[A,J_w]T=T^{-1}J_{C(w)}T=\tilde{J}_{C(w)}.
$$
\end{proof}

Let $x_0\in\mathfrak v_{\min}$ be as in $\bf (H2)$ and $v\in \mathfrak v_{\min}$ as in~\eqref{eq:inv34}.
We find $T\in SO(4,4)$, $T\colon \mathfrak v_{\min}\to \mathfrak v_{\min}$,  such that
$x_0=J_{Z_4}T(v)$. Then condition~\eqref{in nor + +} is equivalent to
\[
\Big(\tilde A+k\Id\Big)\big(\tilde J_{Z_4}(v)\big)=\tilde J_{z_0}\big(\tilde J_{Z_4}(v)\big),\quad \Big(C+k\Id\Big)(z_0)=0,\quad  [\tilde A,\tilde J_w]=\tilde J_{C(w)}
\]
for all $w\in\mathbb{R}^{3,4}$ with $\tilde{A}(v)=T^{-1}AT(v)$ and $\tilde{J}_w$ as in~\eqref{eq:tildeT}. We only show the first equality as follows
\begin{eqnarray*}
\Big(\tilde{A}+k\Id\Big)\big(\tilde J_{Z_4}(v)\big)&=&T^{-1}\Big(A+k\Id\Big)T\big(T^{-1} J_{Z_4}T(v)\big)=T^{-1}\Big(A+k\Id\Big)(x_0)
\\
&=&T^{-1}J_{z_0}(x_0)
=T^{-1}T\tilde J_{z_0}T^{-1}\big(J_{Z_4}T(v)\big)
\\
&=&
 \tilde{J}_{z_0}T^{-1}\big(T\tilde J_{Z_4}T^{-1}T(v)\big)=\tilde J_{z_0}\tilde J_{Z_4}(v).
\end{eqnarray*}
\smallskip

Let $x_0\in\mathfrak v_{\min}$ be as in $\bf (H3)$ and $v\in \mathfrak v_{\min}$ as in~\eqref{eq:inv34}.
We fix an orthonormal basis $\{Z_{i}\}_{i=1}^7$ as in~\eqref{eq:b7}, and express $x_0$ in the basis~\eqref{eq:basis34-7} as follows
$$
x_0=\theta_0v+\sum_{i=1}^{3}\theta_{i}J_{Z_{i}}(v)+\sum_{i=4}^{7}\theta_{i}J_{Z_{i}}(v)
=x_0^{+}+ x_0^{-},\qquad
\theta_0v+\sum_{i=1}^{3}{\theta_{i}}^{2}= \sum_{i=4}^{7}{\theta_{i}}^{2}=1.
$$
Since $\langle x_0^{+}, x_0^-\rangle_{ \mathfrak v_{\min}}=0$, we can find
an orthogonal transformation $T\in SO(4,0)\times SO(0,4)\subset SO(4,4)$
such that $x_0^{+}=T(v)$ and $x_0^{-}=J_{Z_{4}}T(v)$.
Then for $\tilde{J}_{z}=T^{-1}J_{z}T$
we obtain that
\[
\Big(\tilde A+k\Id\Big)\big(v+\tilde J_{Z_4}(v)\big)=\tilde J_{z_0}\big(v+\tilde J_{Z_4}(v)\big),\quad \Big(C+k\Id\Big)(z_0)=0,\quad  [\tilde A,\tilde J_w]=\tilde J_{C(w)}
\]
are equivalent to~\eqref{in nor + +}. We calculate
\begin{eqnarray*}\label{simplified con 1}
\Big(\tilde{A}+k\Id\Big)\big(v+\tilde J_{Z_4}(v)\big)&=&
T^{-1}\Big(A+k\Id \Big)\big(T(v)+TT^{-1}J_{Z_{4}}T(v)\big)
\\
&=&T^{-1}\Big(A+k\Id \Big)\big(x_0^++x_0^-\big)
=T^{-1}J_{z_0}(x_0)
\\
&=&T^{-1}T\tilde J_{z_0}T^{-1}\big(T(v)+J_{Z_4}T(v)\big)
=\tilde J_{z_0}\big(v+\tilde J_{Z_4}(v)\big).
\end{eqnarray*}
The rest of equalities can be shown as was done previously.

We denote the $H$-type group $N_{3,4}(\mathfrak v_{\min})$ related to the Clifford action $J_{z}\colon \mathfrak v_{\min}\to \mathfrak v_{\min}$ and the
$H$-type group $N_{3,4}^T(\mathfrak v_{\min})$ related to the Clifford action $\tilde J_{z}\colon \mathfrak v_{\min}\to \mathfrak v_{\min}$, $\tilde J_{z}=T^{-1}J_zT$. Then the Lie bracket $[x,y]^{T}$ defined by $\tilde J_z$ is given by
\[
\langle z,[x,y]^{T}\rangle_{\mathbb R^{3,4}}=\langle \tilde{J}_{z}(x),y\rangle_{ \mathfrak v_{\min}}.
\]

\begin{lemma} \label{conjugate groups}
The transformation $\Id\times T\colon N_{3,4}(\mathfrak v_{\min})\to N_{3,4}^T(\mathfrak v_{\min})$
is an isometric isomorphism.
\end{lemma}
\begin{proof}
The result follows from the chain of equalities:
\begin{eqnarray*}
\langle z,[T(x),T(y)]\rangle_{\mathbb R^{3,4}}&=&
\langle J_{z}(T(x)),T(y)\rangle_{ \mathfrak v_{\min}}
\\
&=&
\langle T^{\tau}J_{z}T(x),y\rangle_{ \mathfrak v_{\min}}
=\langle \tilde{J}_{z}(x),y\rangle_{ \mathfrak v_{\min}}
=\langle z,[x,y]^{T}\rangle_{ \mathbb R^{3,4}}.
\end{eqnarray*}
\end{proof}
\smallskip

The discussions above can be summarized in the statement that an initial vector $(z_0,x_0)\in\mathfrak z\oplus\mathfrak v_{\min}$ for a geodesic can be restricted to the following nine cases.

\begin{equation}\left.
\begin{array}{llll}
Z_{1}\oplus v,
\\
Z_{4}\oplus v,
\\
(Z_{1}+Z_{4})\oplus v
\\
\\
Z_{1} \oplus J_{Z_{4}}(v),
\\
Z_{4} \oplus J_{Z_{4}}(v),
\\
(Z_{1}+Z_{4})\oplus J_{Z_{4}}(v)
\\
\\
Z_{1}\oplus (v+J_{Z_{4}}(v)),
\\
Z_{4}\oplus (v+J_{Z_{4}}(v)),
\\
(Z_{1}+Z_{4})\oplus (v+J_{Z_{4}}(v))
\end{array}\right\}\in  \mathfrak z\oplus\mathfrak v\cong\mathbb{R}^{3,4}\oplus \mathfrak v_{\min}.
\label{all n}
\end{equation}

We show the  existence of a pair $(C,A)$ for all these $9$ cases~\eqref{all n}, as well as the fact
that they form an affine space of a suitable dimension according to the
initial conditions.

\subsubsection{Relation between $A$ and $C$ in $[A,J_{w}]=J_{C(w)}$}
We express $A$ in the form
\begin{equation}\label{expression of A}
A=\sum_{i<k}\,a_{i,k}J_{Z_{i}}J_{Z_{k}}.
\end{equation}
Indeed, we know that $A$ belongs to linear combination of $\{[J_{Z_{i}},J_{Z_{k}}]\,|\, 1\leq i< k \leq 7\}$, see details in~\cite[Proposition 4]{FuMaNik25}.
On the other hand, $[J_{Z_{i}},J_{Z_{k}}]=J_{Z_{i}}J_{Z_{k}}-J_{Z_{k}}J_{Z_{i}}=2J_{Z_{i}}J_{Z_{k}}$.
Then $[A,J_{Z_1}]=J_{C(Z_1)}$ implies
\begin{equation*}
[A,J_{Z_1}]=A{J_{Z_1}}-J_{Z_1}A=\sum_{i<k} a_{i,k}J_{Z_i}J_{Z_k}J_{Z_1}
-J_{Z_1}\sum_{i<k} a_{i,k}J_{Z_i}J_{Z_k}
=2\sum_{k>1} a_{1,k}J_{Z_k}=J_{2\sum_{k=2}^{7} a_{1,k}Z_{k}}.
\end{equation*}
Thus $C(Z_1)=2\sum_{k=2}^{7} a_{1,k}Z_{k}$.
By similar calculations we obtain
\begin{equation*}
\begin{array}{lllll}
%&C(Z_{1})=2\sum_{k=2}^{7} a_{1,k} Z_{k},\quad
&C(Z_{2})=-2a_{1,2} Z_{1}+ 2\sum_{k=3}^7 a_{2,k}Z_{k},\quad
&C(Z_{3})=-2\sum_{k=1}^{2} a_{k,3} Z_{k}+ 2\sum_{k=4}^{7} a_{3,k}Z_{k},\\
&C(Z_{4})=2\sum_{k=1}^{3} a_{k,4} Z_{k}- 2\sum_{k=5}^{7} a_{4,k}Z_{k},\quad
&C(Z_{5})=2\sum_{k=1}^{4} a_{k,5} Z_{k}- 2\sum_{k=6}^{7} a_{5,k}Z_{k},\\
&C(Z_{6})=2\sum_{k=1}^{5} a_{k,6} Z_{k}- 2a_{6,7}Z_{7},\quad
&C(Z_{7})=2\sum_{k=1}^{6} a_{k,7} Z_{k}.
\end{array}
\end{equation*}
In the matrix form $C$ is given by
\begin{equation}C=2\times
\begin{pmatrix}
0&-a_{1,2}&-a_{1,3}&a_{1,4}&a_{1,5}&a_{1,6}&a_{1,7}\\
a_{1,2}&0&-a_{2,3}&a_{2,4}&a_{2,5}&a_{2,6}&a_{2,7}\\
a_{1,3}&a_{2,3}&0&a_{3,4}&a_{3,5}&a_{3,6}&a_{3,7}\\
a_{1,4}&a_{2,4}&a_{3,4}&0&a_{4,5}&a_{4,6}&a_{4,7}\\
a_{1,5}&a_{2,5}&a_{3,5}&-a_{4,5}&0&a_{5,6}&a_{5,7}\\
a_{1,6}&a_{2,6}&a_{3,6}&-a_{4,6}&-a_{5,6,}&0&a_{6,7}\\
a_{1,7}&a_{2,7}&a_{3,7}&-a_{4,7}&-a_{5,7}&-a_{6,7}&0
\end{pmatrix}.\label{M expression of C}
\end{equation}

\subsubsection{Writing $A(v)$ and $A(J_{Z_{4}}(v))$ in basis~\eqref{eq:basis34-7} }

Note that condition~\eqref{eq:inv34} implies
\begin{equation}\label{eq:1=2}
    \begin{array}{llllllllll}
J_{Z_1}(v)&=&-J_{Z_2}J_{Z_3}(v)&=&J_{Z_4}J_{Z_7}(v)&=&J_{Z_5}J_{Z_6}(v),\quad
\\
J_{Z_2}(v)&=&J_{Z_1}J_{Z_3}(v)&=&J_{Z_4}J_{Z_6}(v)&=&-J_{Z_5}J_{Z_7}(v),\quad
\\
J_{Z_3}(v)&=&-J_{Z_1}J_{Z_2}(v)&=&J_{Z_4}J_{Z_5}(v)&=&J_{Z_6}J_{Z_7}(v),
\\
J_{Z_4}(v)&=&J_{Z_1}J_{Z_7}(v)&=&J_{Z_2}J_{Z_6}(v)&=&J_{Z_3}J_{Z_5}(v),
\\
J_{Z_5}(v)&=&-J_{Z_2}J_{Z_7}(v)&=&-J_{Z_3}J_{Z_4}(v)&=&J_{Z_1}J_{Z_6}(v),\quad
\\
J_{Z_6}(v)&=&-J_{Z_1}J_{Z_5}(v)&=&-J_{Z_2}J_{Z_4}(v)&=&J_{Z_3}J_{Z_7}(v),
\\
J_{Z_7}(v)&=&-J_{Z_1}J_{Z_4}(v)&=&J_{Z_2}J_{Z_5}(v)&=&-J_{Z_3}J_{Z_6}(v).
\end{array}
\end{equation}

Put
$$A(v)=\sum_{i<k}\,a_{i,k}J_{Z_{i}}J_{Z_{k}}(v)=\alpha_0 v+\sum_{j=1}^{7}\,\alpha_{j}J_{Z_{j}}(v)=\sum_{j=0}^7 \alpha_jX_{j+1}.
$$
Then using~\eqref{eq:1=2} we find the relations between coefficients $a_{i,k}$ and $\alpha_j$ as follows

\begin{equation}
\left\{
\begin{array}{l}
\alpha_{0}=0,\\
\alpha_{1}=-a_{2,3}+a_{4,7}+a_{5,6},\\%=1,\\
\alpha_{2}=  a_{1,3}+a_{4,6}-a_{5,7},\\%=0,\\
\alpha_{3}= -a_{1,2}+a_{4,5}+a_{6,7},\\%=0,\\
\alpha_{4}= a_{1,7}+a_{2,6}+a_{3,5},\\%=0,\\
\alpha_{5}= a_{1,6}-a_{2,7}-a_{3,4},\\%=0,\\
\alpha_{6}= -a_{1,5}-a_{2,4} + a_{3,7},\\%=0,\\
\alpha_{7}= -a_{1,4}+ a_{2,5}-a_{3,6}.\\%=0.
\end{array}\right.
\label{Av*}
\end{equation}
Write
$$
A(J_{Z_{4}}(v))=\sum_{i<k}\,a_{i,k}J_{Z_{i}}J_{Z_{k}}(v)=
\beta_0 v+\sum_{m=1}^{7}\beta_{m}J_{Z_{m}}(v)=\sum_{m=0}^7 \beta_mX_{m+1},
$$
and note that
\begin{equation}\label{eq:1=3}
    \begin{array}{llllllllll}
J_{Z_1}(v)&=&J_{Z_1}J_{Z_4}J_{Z_4}(v)&=&J_{Z_2}J_{Z_5}J_{Z_4}(v)&=&-J_{Z_3}J_{Z_6}J_{Z_4}(v)
\\
J_{Z_2}(v)&=&-J_{Z_1}J_{Z_5}J_{Z_4}(v)&=&J_{Z_2}J_{Z_4}J_{Z_4}(v)&=&J_{Z_3}J_{Z_7}J_{Z_4}(v)
\\
J_{Z_3}(v)&=&J_{Z_1}J_{Z_6}J_{Z_4}(v)&=&-J_{Z_2}J_{Z_7}J_{Z_4}(v)&=&J_{Z_3}J_{Z_4}J_{Z_4}(v)
\\
v&=&J_{Z_1}J_{Z_7}J_{Z_4}(v)&=&J_{Z_2}J_{Z_6}J_{Z_4}(v)&=&J_{Z_3}J_{Z_5}J_{Z_4}(v)
\\
J_{Z_5}(v)&=&-J_{Z_1}J_{Z_2}J_{Z_4}(v)&=&-J_{Z_4}J_{Z_5}J_{Z_4}(v)&=&J_{Z_6}J_{Z_7}J_{Z_4}(v)
\\
J_{Z_6}(v)&=&J_{Z_1}J_{Z_3}J_{Z_4}(v)&=&-J_{Z_4}J_{Z_6}J_{Z_4}(v)&=&-J_{Z_5}J_{Z_7}J_{Z_4}(v)
\\
J_{Z_7}(v)&=&-J_{Z_2}J_{Z_3}J_{Z_4}(v)&=&-J_{Z_4}J_{Z_7}J_{Z_4}(v)&=&J_{Z_5}J_{Z_6}J_{Z_4}(v).
\end{array}
\end{equation}

By using~\eqref{eq:1=3} we find relations between $a_{i,k}$ and $\beta_j$:
\begin{equation}
\left\{
\begin{array}{ll}
\beta_{0}=a_{1,7}+a_{2,6}+a_{3,5},\\%=0,\\
\beta_{1}=a_{1,4}+a_{2,5}-a_{3,6},\\%=0,\\
\beta_{2}=-a_{1,5}+a_{2,4}+a_{3,7},\\%=0,\\
\beta_{3}=a_{1,6}-a_{2,7}+a_{3,4},\\%=0,\\
\beta_{4}=0,\\%=0,\\
\beta_{5}=-a_{1,2} -a_{4,5}+a_{6,7},\\%=0,\\
\beta_{6}=a_{1,3}-a_{4,6}-a_{5,7},\\%=0,\\
\beta_{7}=-a_{2,3}-a_{4,7}+a_{5,6}.%=-1,
\end{array}\right.\label{AJ4v*}
\end{equation}

\subsubsection{Existence of $D=(C,A)$}\label{sec:final}

Based on data in the previous section, we discuss the existence
of a pair $(C,A)$ of condition~\eqref{in nor + +}
for each cases of initial conditions~\eqref{all n}.

{\sc Case $\|x_0\|^2=\|z_0\|^2=1$}. Since the initial vector is not null, it is enough to check the conditions $C(Z_{1})=0$, $A(v)=J_{Z_{1}}(v)$ by Lemma~\ref{lem:1}. We use~\eqref{M expression of C},~\eqref{Av*} and obtain
\begin{align}
&a_{1,2}=a_{1,3}=a_{1,4}=a_{1,5}=a_{1,6}=a_{1,7}=0,~\text{and}\nonumber
\\
&\left\{
\begin{array}{l}
\alpha_{0}=0\\
\alpha_{1}=-a_{2,3}+a_{4,7}+a_{5,6}=1\\
\alpha_{2}=  a_{1,3}+a_{4,6}-a_{5,7}=0\\
\alpha_{3}= -a_{1,2}+a_{4,5}+a_{6,7}=0\\
\alpha_{4}= a_{1,7}+a_{2,6}+a_{3,5}=0\\
\alpha_{5}= a_{1,6}-a_{2,7}-a_{3,4}=0\\
\alpha_{6}= -a_{1,5}-a_{2,4} + a_{3,7}=0\\
\alpha_{7}= -a_{1,4}+ a_{2,5}-a_{3,6}=0,
\end{array}\right.
\implies
\left\{
\begin{array}{l}
a_{5,6}=1+a_{2,3}-a_{4,7}\\
a_{5,7}=a_{4,6}\\
a_{6,7}=-a_{4,5}\\
a_{3,5}=-a_{2,6}\\
a_{3,4}=-a_{2,7}\\
a_{3,7}=a_{2,4} \\
a_{3,6}=a_{2,5},
\end{array}\right.
\\
&a_{2,3},a_{2,4},a_{2,5},a_{2,6},a_{2,7},\ a_{4,5},a_{4,6},a_{4,7}\quad\text{are free parameters}.\nonumber
\label{Av}
\end{align}

{\sc Case $\|x_0\|^2=-\|z_0\|^2=1$}. In this case the initial vector is null and we need to check the conditions $\Big(C+k\Id\Big)(Z_{4})=0$, $\Big(A+k\Id\Big)(v)=J_{Z_{4}}(v)$ according to Lemma~\ref{lem:1}. The condition $\Big(C+k\Id\Big)(Z_{4})=0$ gives
\begin{equation}\label{eq:Z4}
a_{1,4}=a_{2,4}=a_{3,4}=a_{4,5}=a_{4,6}=a_{4,7}=k=0.
\end{equation}
Thus, we only need to check $A(v)=J_{Z_{4}}(v)$, which leads to
\begin{align}
&\begin{cases}
\alpha_{0}=0,\\
\alpha_{1}=-a_{2,3}+a_{4,7}+a_{5,6}=0,\\
\alpha_{2}=  a_{1,3}+a_{4,6}-a_{5,7}=0,\\%=0,\\
\alpha_{3}= -a_{1,2}+a_{4,5}+a_{6,7}=0,\\%=0,\\
\alpha_{4}= a_{1,7}+a_{2,6}+a_{3,5}=1,\\%=0,\\
\alpha_{5}= a_{1,6}-a_{2,7}-a_{3,4}=0,\\%=0,\\
\alpha_{6}= -a_{1,5}-a_{2,4} + a_{3,7}=0,\\%=0,\\
\alpha_{7}= -a_{1,4}+ a_{2,5}-a_{3,6}=0,\\%=0.
\end{cases}
\quad\implies\quad
\begin{cases}
a_{5,6}=a_{2,3},\\%=1,\\
a_{5,7}=a_{1,3},\\%=0,\\
a_{6,7}=a_{1,2},\\%=0,\\
a_{3,5}=1-a_{1,7}-a_{2,6},\\%=0,\\
a_{2,7}=a_{1,6},\\%=0,\\
a_{3,7}=a_{1,5},\\%=0,\\
a_{3,6}=a_{2,5},\\%=0.
\end{cases}
\\
&\quad a_{1,2}, a_{1,3},a_{1,5},a_{1,6},a_{1,7},a_{2,3},a_{2,5},a_{2,6}\quad\text{are free parameters}.\nonumber
\label{Av*1}
\end{align}
by~\eqref{Av*}.
\smallskip
Analogously to the above case, one can obtain solutions for any of the following initial conditions:
\begin{equation}\label{eq: 111}
\begin{array}{cccc}
&Z_{1}\oplus v, \quad Z_{4}\oplus v, \quad
Z_{1}\oplus J_{Z_{4}}(v), \quad Z_{4}\oplus J_{Z_{4}}(v),
\\
&(Z_{1}+Z_{4})\oplus v, \qquad (Z_{1}+Z_{4})\oplus J_{Z_{4}}(v).
\end{array}
\end{equation}

{\sc Case $(z_0,x_0)$ with $\|x_0\|=0$ and/or $\|z_0\|=0$.}
\quad
The cases $Z_{1}\oplus (v+J_{Z_{4}}(v))$, $Z_{4}\oplus (v+J_{Z_{4}}(v))$, and
$(Z_{1}+Z_{4})\oplus (v+J_{Z_{4}}(v))$
are treated by a similar way. We consider only the
$(Z_{1}+Z_{4})\oplus (v+J_{Z_{4}}(v))$ in some details.

Conditions~\eqref{in nor + +} are equivalent to
\begin{equation*}
\begin{array}{lllll}
\Big(A+k\Id\Big)(v+J_{Z_{4}}(v))&=&A(v+J_{Z_{4}}(v))+kv+kJ_{Z_4}v
\\
&=&J_{Z_{1}+Z_{4}}(v+J_{Z_{4}}(v))
\\
&=&J_{Z_{1}}(v)+J_{Z_1Z_{4}}(v)+J_{Z_{4}}(v) +v
\\
&=&X_1+X_2+X_5-X_8,
\end{array}
\end{equation*}
by~\eqref{eq:basis34-7}.
We obtain the system
\begin{equation}\label{eq:gamma}
\left\{
\begin{array}{llll}
&\alpha_{0}+\beta_{0}= a_{1,7}+a_{2,6}+a_{3,5}=1-k\\
&\alpha_{1}+\beta_{1}=(-a_{2,3}+a_{5,6}-a_{3,6}+a_{2,5}) +(a_{1,4}+a_{4,7})=1\\
&\alpha_{2}+\beta_{2} =(a_{1,3}-a_{1,5}+a_{3,7}-a_{5,7})+(a_{4,6}+a_{2,4})=0\\
&\alpha_{3}+\beta_{3}=(-a_{1,2}+a_{6,7}+a_{1,6}-a_{2,7})+(a_{4,5}+a_{3,4})=0\\
&\alpha_{4}+\beta_{4}=a_{1,7}+a_{2,6}+a_{3,5}=1-k\\
&\alpha_{5}+\beta_{5}=(-a_{1,2}+a_{6,7}+a_{1,6}-a_{2,7})-(a_{4,5}+a_{3,4})=0\\
&\alpha_{6}+\beta_{6}= (a_{1,3}-a_{1,5}+a_{3,7}-a_{5,7})-(a_{4,6}+a_{2,4}) =0\\
&\alpha_{7}+\beta_{7}= (-a_{2,3}+a_{5,6} -a_{3,6}+a_{2,5}) -(a_{1,4}+a_{4,7})=-1.
\end{array}
\right.
\end{equation}
The condition $C(Z_{1}+Z_{4})=0$ leads to
$$
a_{1,4}=-k, ~a_{1,2}+a_{2,4}=0,~a_{1,3}+a_{3,4}=0,
~a_{1,5}-a_{4,5}=0,~a_{1,6}-a_{4,6}=0,~a_{1,7}-a_{4,7}=0
$$
From these expressions we obtain the solution to~\eqref{eq:gamma}:
\begin{equation*}
\begin{cases}
a_{1,4}=-k/2
\\
a_{4,6}=-a_{2,4}=a_{1,6}=a_{1,2}
\\
a_{4,5}=-a_{3,4}=a_{1,5}=a_{1,3}
\\
a_{4,7}=a_{1,7}=1+k/2
\\
a_{5,6}=a_{2,3}-a_{2,5}+a_{3,6}
\\
a_{3,5}=-a_{2,6}-3k/2
\\
a_{6,7}=a_{2,7}
\\
a_{5,7}=a_{3,7}
\end{cases}
\end{equation*}
with eight free parameters $a_{1,2},a_{1,3},a_{2,3},a_{2,5},a_{2,6},a_{2,7},a_{3,6},a_{3,7}$ and one additional $k$.

This finishes the proof of Theorem~\ref{th:34min}.
We emphasize that this proof is constructive and it is easy to find a 1-parameter group
(and even a multi-parameter family of such groups) whose orbit is an arbitrarily chosen geodesic.

\begin{remark}
For each initial vector, the set of pairs $(D,k)$ solving~\eqref{in nor + +} is a non-empty affine
subspace. Its dimension is not always $8$. It equals $8$ for the cases $Z_1\oplus v$ and
$Z_4\oplus v$ and for generic initial vectors. It equals $9$ for
$(Z_1+Z_4)\oplus(v+J_{Z_4}v)$, where $k$ is free, and $11$ for $Z_4\oplus(v+J_{Z_4}v)$. The latter one easily follows from~\eqref{eq:Z4} for $k=0$ and~\eqref{eq:gamma}.
\end{remark}

\subsection{Pseudo-Riemannian  $H$-type nilmanifold ${N}_{3,4}(\mathfrak v_{\min}\oplus \mathfrak v_{\min})$}\label{sec:34--}

Now we  consider a Clifford module $\mathfrak v^{3,4}$ of $\Cl(\mathbb R^{3,4})$ which is not minimal admissible but rather the orthogonal sum $\mathfrak v^{3,4}=\mathfrak v_{\min}\oplus \mathfrak v_{\min}$ of two equivalent minimal admissible modules. The module $\mathfrak v^{3,4}$ is isometric to $\mathbb R^{8,8}$.

\begin{theorem}\label{th:34++}
The pseudo-Riemannian $H$-type nilmanifold $N_{3,4}(\mathfrak v_{\min}\oplus \mathfrak v_{\min})$ is not a geodesic orbit manifold.
\end{theorem}
\begin{proof}
%We can assume that the decomposition $\mathfrak v=\mathfrak v_{\min}\oplus \mathfrak v_{\min}$ is orthogonal.
We will find an initial vector $(z_0,x_0,y_0)$ for a geodesic with $(x_0,y_0)=x_0\oplus y_0\in \mathfrak v_{\min}\oplus \mathfrak v_{\min}$, which can not be extended to an element $D\in\mathfrak h$.

We fix an orthonormal basis $\{Z_1,\ldots,Z_7\}$ for $\mathfrak z=\mathbb R^{3,4}$, choose the null vector $z_0=Z_1+Z_4$ and unit vectors $v=v\oplus 0,u=0\oplus v$, where $v\in \mathfrak v_{\min}$ is the eigenvector of involutions~\eqref{eq:inv34}. Then, we also assume
$$
x_0=v+J_{Z_{7}}(v),\quad
y_0=-J_{Z_{2}}(v)+J_{Z_{5}}(v).
$$
Conditions~\eqref{in nor + +} are written in the form
\begin{equation}\label{eq:A1A2}
\begin{pmatrix}A_1+k\Id&B_{1}\\B_{2}&A_2+k\Id\end{pmatrix}
\begin{pmatrix} x_0\\y_0\end{pmatrix}
=
\begin{pmatrix}J_{z_0}&0\\0&J_{z_0}\end{pmatrix}
\begin{pmatrix}x_0\\ y_0\end{pmatrix},
\end{equation}
\begin{equation*}\label{eq:A1A2-1}
\left[\begin{pmatrix}A_1&B_{1}\\B_{2}&A_2\end{pmatrix}
\begin{pmatrix}J_{w}&0\\0&J_{w}\end{pmatrix}\right]
=\begin{pmatrix} [A_1,J_{w}] & [B_{1},J_{w}] \\ [B_{2},J_{w}] & [A_2,~J_{w}] \end{pmatrix}
=\begin{pmatrix}
J_{C(w)}&0\\
0&J_{C(w)}
\end{pmatrix},
\end{equation*}
for all $w\in\mathbb{R}^{3,4}$. Then $B_{1}=\beta\cdot \Id$ and $B_{2}=-\beta\cdot \Id$ and $A_2=A_1=A$, with $A^{\tau}=-A$ by the skew symmetry of $D$.
Thus, equation~\eqref{eq:A1A2} takes the form

\begin{equation*}
\begin{array}{llll}
&\Big(A+k\Id\Big)(x_0)+\beta y_0=J_{Z_1+Z_4}(x_0),
\\
&-\beta x_0+\Big(A+k\Id\Big)(y_0)=J_{Z_{1}+Z_{4}}(y_0),
\\
&\Big(C+k\Id\Big)(Z_{1}+Z_4)=0,
\\
&[A, J_{w}]=J_{C(w)},\quad\text{for all}\quad w\in\mathfrak z=\mathbb R^{3,4}
\end{array}
\end{equation*}
To show that a geodesic with the initial conditions $(z_0,x_0,y_0)=(Z_1+Z_4,v+J_{Z_7}(v), -J_{Z_{2}}(v)+J_{Z_{5}}(v))$ is homogeneous, we need to find a pair $(C, A)\in \mathfrak h$ and $\beta\in\mathbb{R}$ satisfying the above relations.
The condition $\Big(C+k\Id\Big)(Z_{1}+Z_{4})=0$ leads to
\begin{equation}\label{eq:C14}
a_{1,4}=-k/2, ~a_{1,2}+a_{2,4}=0,~a_{1,3}+a_{3,4}=0,
~a_{1,5}-a_{4,5}=0,~a_{1,6}-a_{4,6}=0,~a_{1,7}-a_{4,7}=0.
\end{equation}
We also find that
$$
J_{Z_1+Z_4}(x_0)=J_{Z_1+Z_4}(v+J_{z_{7}}(v))=2J_{Z_1}(v)+2J_{Z_4}(v),\quad
$$
$$
J_{Z_1+Z_4}(y_0)=J_{Z_1+Z_4}(-J_{Z_{2}}(v)+J_{Z_{5}}(v))=2J_{Z_3}(v)-2J_{Z_6}(v),\quad
$$

We write the values
$$A(v),\quad A\big(J_{Z_7}(v)\big), \quad A\big(J_{Z_2}(v)\big), \quad A\big(J_{Z_5}(v)\big)$$ in basis~\eqref{eq:basis34-7}.
Put $A(v)=\sum_{i<m}\,a_{i,m}J_{Z_{i}}J_{Z_{m}}(v)=\alpha_0 v+\sum_{j=1}^{7}\,\alpha_{j}J_{Z_{j}}(v)$. The coefficients $\alpha_j$ is given in~\eqref{Av*} and we rewrite them by using~\eqref{eq:C14}. Analogously, we calculate
$$
%A(J_{Z_{3}}(v))=\varphi_0 v+\sum_{i=1}^7 \varphi_{i}J_{Z_{i}}(v)
A(J_{Z_{7}}(v))=\pi_0 v+\sum_{i=1}^7 \pi_{i}J_{Z_{i}}(v),\quad
A(J_{Z_{2}}(v))=\omega_0 v+\sum_{i=1}^7 \omega_{i}J_{Z_{i}}(v),
$$
$$
A(J_{Z_{5}}(v))=\delta_0 v+\sum_{i=1}^7 \delta_{i}J_{Z_{i}}(v)
$$
\begin{equation}
\left\{
\begin{array}{l}
\alpha_{0}=0,\\
\alpha_{1}=-a_{2,3}+a_{4,7}+a_{5,6},\\
\alpha_{2}=  a_{1,3}+a_{4,6}-a_{5,7}=-a_{3,4}+a_{4,6}-a_{5,7},\\%=0,\\
\alpha_{3}= -a_{1,2}+a_{4,5}+a_{6,7}=a_{2,4}+a_{4,5}+a_{6,7},\\%=0,\\
\alpha_{4}= a_{1,7}+a_{2,6}+a_{3,5}=a_{2,6}+a_{3,5}+a_{4,7},\\%=0,\\
\alpha_{5}= a_{1,6}-a_{2,7}-a_{3,4}=-a_{2,7}-a_{3,4}+a_{4,6},\\%=0,\\
\alpha_{6}= -a_{1,5}-a_{2,4} + a_{3,7}=-a_{2,4} + a_{3,7}-a_{4,5},\\%=0,\\
\alpha_{7}= -a_{1,4}+ a_{2,5}-a_{3,6}=a_{2,5}-a_{3,6}+k/2.\\%=0.
\end{array}\right.\label{Avv}
\end{equation}

\begin{equation}
\left\{
\begin{array}{l}
\pi_{0}=-a_{1,4}+a_{2,5}-a_{3,6}=a_{2,5}-a_{3,6}+k/2\\
\pi_{1}=a_{1,7}-a_{2,6}-a_{3,5}=-a_{2,6}-a_{3,5}+a_{4,7},\\%=1,\\
\pi_{2}=a_{1,6}+a_{2,7}-a_{3,4}=a_{2,7}-a_{3,4}+a_{4,6},\\%=0,\\
\pi_{3}=a_{1,5}+ a_{2,4}+a_{3,7}=a_{2,4}+a_{3,7}+a_{4,5}\\%=0,\\
\pi_{4}= a_{2,3}+a_{4,7}-a_{5,6},\\%=0,\\,\\%=0,\\
\pi_{5}= a_{1,3}+a_{4,6} + a_{5,7}=-a_{3,4}+a_{4,6} + a_{5,7},\\%=0,\\
\pi_{6}= a_{1,2}-a_{4,5}+a_{6,7}=-a_{2,4}-a_{4,5}+a_{6,7},,\\%=0,\\
\pi_{7}=0 .\\%=0.
\end{array}\right.\label{Avpi}
\end{equation}

\begin{equation}
\left\{
\begin{array}{l}
\omega_{0}=-a_{1,3}-a_{4,6}+a_{5,7}=a_{3,4}-a_{4,6}+a_{5,7},\\
\omega_{1}=-a_{1,2}-a_{4,5}-a_{6,7}=a_{2,4}-a_{4,5}-a_{6,7},,\\%=1,\\
\omega_{2}=0,\\%=0,\\
\omega_{3}=a_{2,3}+a_{4,7}+a_{5,6} ,\\%=0,\\
\omega_{4}=-a_{1,5}+ a_{2,4}+a_{3,7}=a_{2,4}+a_{3,7} -a_{4,5} ,\\%=0,\\
\omega_{5}= a_{1,4}+a_{2,5} + a_{3,6}=a_{2,5} + a_{3,6}-k/2,\\%=0,\\
\omega_{6}= -a_{1,7}+a_{2,6}-a_{3,5}=a_{2,6}-a_{3,5}-a_{4,7},\\%=0,\\
\omega_{7}=a_{1,6}+a_{2,7}-a_{3,4}=a_{2,7}-a_{3,4}+a_{4,6} .\\%=0.
\end{array}\right.\label{Avomega}
\end{equation}

\begin{equation}
\left\{
\begin{array}{l}
\delta_{0}=a_{1,6}-a_{2,7}-a_{3,4}=-a_{2,7}-a_{3,4}+a_{4,6},\\
\delta_{1}=a_{1,5}-a_{2,4}+a_{3,7}=-a_{2,4}+a_{3,7}+a_{4,5},\\
\delta_{2}=a_{1,4}+a_{2,5}+a_{3,6}=a_{2,5}+a_{3,6}-k/2,\\
\delta_{3}=-a_{1,7}-a_{2,6}+a_{3,5}=-a_{2,6}+a_{3,5}-a_{4,7},\\
\delta_{4}=a_{1,2}+a_{4,5}-a_{6,7}=-a_{2,4}+a_{4,5}-a_{6,7},\\
\delta_{5}=0,\\
\delta_{6}=-a_{2,3}+a_{4,7}-a_{5,6},\\
\delta_{7}=-a_{1,3}-a_{4,6}-a_{5,7}=a_{3,4}-a_{4,6}-a_{5,7}.
\end{array}\right.\label{AJ5v*}
\end{equation}

The equation
$$
\Big(A+k\Id\Big)(v+J_{Z_7}(v))+\beta (-J_{Z_2}(v)+J_{Z_5}(v))=J_{Z_1+Z_4}(v+J_{Z_7(v)})=2J_{Z_1}(v)+2J_{Z_4}(v)
$$
leads to the system
\begin{equation*}
\begin{cases}
(0)\quad a_{2,5}-a_{3,6}+k=-k
\\
(1)\quad -a_{2,3}-a_{2,6}-a_{3,5}+2a_{4,7}+a_{5,6}=2
\\
(2)\quad a_{2,7}-2a_{3,4}+2a_{4,6}-a_{5,7}=\beta
\\
(3)\quad 2a_{2,4}+a_{3,7}+2a_{4,5}+a_{6,7}=0
\\
(4)\quad a_{2,3}+a_{2,6}+a_{3,5}+2a_{4,7}-a_{5,6}=2
\\
(5)\quad -a_{2,7}-2a_{3,4}+2a_{4,6}+a_{5,7}=-\beta
\\
(6)\quad -2a_{2,4}+a_{3,7}-2a_{4,5}+a_{6,7}=0
\\
(7)\quad a_{2,5}-a_{3,6}+k=-k,
\end{cases}
\end{equation*}
by using~\eqref{Avv} and~\eqref{Avpi}. Summing equations $(1)$ and $(4)$, we obtain $a_{4,7}=1$.  The equation
$$
\Big(A+k\Id\Big)\big(-J_{Z_2}(v)+J_{Z_5}(v)\big)-\beta \big(v+J_{Z_7}(v)\big)=J_{Z_1+Z_4}\big(-J_{Z_2}(v)+J_{Z_5}(v)\big)=2J_{Z_3}(v)-2J_{Z_6}(v)
$$
leads to the system
\begin{equation*}
\begin{cases}
(0)\quad -a_{2,7}-2a_{3,4}+2a_{4,6}-a_{5,7}=\beta
\\
(1)\quad -2a_{2,4}+a_{3,7}+2a_{4,5}+a_{6,7}=0
\\
(2)\quad a_{2,5}+a_{3,6}-k=k
\\
(3)\quad -a_{2,3}-a_{2,6}+a_{3,5}-2a_{4,7}-a_{5,6}=2
\\
(4)\quad -2a_{2,4}-a_{3,7}+2a_{4,5}-a_{6,7}=0
\\
(5)\quad -a_{2,5}-a_{3,6}+k=-k
\\
(6)\quad -a_{2,3}-a_{2,6}+a_{3,5}+2a_{4,7}-a_{5,6}=-2
\\
(7)\quad -a_{2,7}+2a_{3,4}-2a_{4,6}-a_{5,7}=\beta,
\end{cases}
\end{equation*}
by using~\eqref{Avomega} and~\eqref{AJ5v*}. Subtracting from equation $(3)$ equation $(6)$, we obtain $a_{4,7}=-1$. Thus we can not solve both equations
$$
\Big(A+k\Id\Big)(x_0)+\beta y_0=J_{Z_1+Z_4}(x_0),\qquad -\beta x_0+\Big(A+k\Id\Big)(y_0)=J_{Z_{1}+Z_{4}}(y_0),
$$
simultaneously in order to find a map $D\in\mathfrak h$.
\end{proof}

\begin{remark}
Since the admissible scalar product on $\mathfrak v_{\min}$ is unique only up to a non-zero
factor, $\mathfrak v_{\min}\oplus \mathfrak v_{\min}$ may carry $\ip{\cdot}{\cdot}\oplus\sigma\ip{\cdot}{\cdot}$ with
$\sigma=\pm1$. Skew-symmetry of $D$ then gives $B_1=\beta\Id$ and $B_2=-\sigma\beta\Id$ in~\eqref{eq:A1A2}. The equations used to derive the contradiction are the same for both values of $\sigma$, so
Theorem~\ref{th:34++} holds for both metrics.
\end{remark}

\subsection{Pseudo-Riemannian  $H$-type nilmanifold ${N}_{3,4}(\mathfrak v_{\min}^{-}\oplus \mathfrak v_{\min}^{+})$}\label{sec:34+-}

 Let $\mathfrak v=\mathfrak v_{\min}^{-}\oplus \mathfrak v_{\min}^{+}$ be an admissible Clifford module composed of two non-equivalent modules of $\Cl(\mathbb R^{3,4})$. Let $\Omega=Z_1\cdots Z_7$ be the volume form for $\Cl(\mathbb R^{3,4})$. Then $\Omega^2=\mathbf{ 1}\in \Cl(\mathbb R^{3,4})$, and
$$
J_{\omega}(x)=x\quad\text{for}\quad x\in \mathfrak v_{\min}^{+},\qquad
J_{\omega}(x)=-x\quad\text{for}\quad x\in \mathfrak v_{\min}^{-}.
$$
The admissible Clifford module $\mathfrak v=\mathfrak v_{\min}^-\oplus \mathfrak v_{\min}^{+}$ is isometric to $\mathbb R^{8,8}$.

\begin{theorem}\label{th:34+-}
The pseudo-Riemannian $H$-type nilmanifold $N_{3,4}(\mathfrak v_{\min}^-\oplus \mathfrak v_{\min}^+)$ is not a geodesic orbit manifold.
\end{theorem}
\begin{proof}
We choose $v^-\in \mathfrak v_{\min}^-$ satisfying~\eqref{eq:inv34} and basis for $\mathfrak v_{\min}^-$ as in~\eqref{eq:basis34-7}. Then we take $v^+\in \mathfrak v_{\min}^+$ satisfying
\begin{equation}\label{eq:inv34--}
\begin{array}{llll}
&J_{Z_1}J_{Z_2}J_{Z_4}J_{Z_5}(v^+)=v^+,\
&J_{Z_1}J_{Z_2}J_{Z_6}J_{Z_7}(v^+)=v^+,\
\\
&J_{Z_1}J_{Z_3}J_{Z_5}J_{Z_7}(v^+)=v^+,\
&J_{Z_1}J_{Z_2}J_{Z_3}(v^+)=-v^+,
\end{array}
\end{equation}
and the basis for $\mathfrak v_{\min}^+$ as in~\eqref{eq:basis34-7}. We take the vector $(z_0,x_0)=(Z_1,v^-\oplus v^+)$ being an initial vector of a geodesic and show, that there is no $D\in\mathfrak h$ satisfying~\eqref{in nor + +}. As in Section~\ref{sec:34--} we need to solve the system

\begin{equation*}
\begin{pmatrix}A+k\Id&\beta\Id\\-\beta\Id&A+k\Id\end{pmatrix}
\begin{pmatrix} v^-\\v^+\end{pmatrix}
=
\begin{pmatrix}
J_{Z_1}v^-\\J_{Z_1}v^+
\end{pmatrix},
\end{equation*}
or
\begin{equation}\label{eq:A1A2-1}
\begin{array}{llll}
&\Big(A+k\Id\Big)(v^-)+\beta v^+=J_{Z_1}(v^-),
\\
&-\beta v^-+\Big(A+k\Id\Big)(v^+)=J_{Z_{1}}(v^+),
\\
&\Big(C+k\Id\Big)(Z_{1})=0,
\end{array}
\end{equation}
The last equation implies that $k=0$. We see also that $\beta=0$ since the modules are not equivalent.
To calculate
$$A(v^+)=\sum_{i<k}\,a_{i,k}J_{Z_{i}}J_{Z_{k}}(v^+)=\tilde\alpha_0 v^++\sum_{j=1}^{7}\,\tilde\alpha_{j}J_{Z_{j}}(v^+),
$$
we need to use instead of~\eqref{eq:1=2}, the following conditions implied by~\eqref{eq:inv34--}
\begin{equation}\label{eq:1=2--}
    \begin{array}{llllllllll}
J_{Z_1}(v^+)&=&J_{Z_2}J_{Z_3}(v^+)&=&-J_{Z_4}J_{Z_7}(v^+)&=&-J_{Z_5}J_{Z_6}(v^+),\quad
\\
J_{Z_2}(v^+)&=&-J_{Z_1}J_{Z_3}(v^+)&=&-J_{Z_4}J_{Z_6}(v^+)&=&J_{Z_5}J_{Z_7}(v^+),\quad
\\
J_{Z_3}(v^+)&=&J_{Z_1}J_{Z_2}(v^+)&=&-J_{Z_4}J_{Z_5}(v^+)&=&-J_{Z_6}J_{Z_7}(v^+),
\\
J_{Z_4}(v^+)&=&-J_{Z_1}J_{Z_7}(v^+)&=&-J_{Z_2}J_{Z_6}(v^+)&=&-J_{Z_3}J_{Z_5}(v^+),
\\
J_{Z_5}(v^+)&=&J_{Z_2}J_{Z_7}(v^+)&=&J_{Z_3}J_{Z_4}(v^+)&=&-J_{Z_1}J_{Z_6}(v^+),\quad
\\
J_{Z_6}(v^+)&=&J_{Z_1}J_{Z_5}(v^+)&=&J_{Z_2}J_{Z_4}(v^+)&=&-J_{Z_3}J_{Z_7}(v^+),
\\
J_{Z_7}(v^+)&=&J_{Z_1}J_{Z_4}(v^+)&=&-J_{Z_2}J_{Z_5}(v^+)&=&J_{Z_3}J_{Z_6}(v^+).
\end{array}
\end{equation}
Thus the solution to~\eqref{eq:A1A2-1} is equivalent to the solution of the system
\begin{equation}
\begin{cases}
\alpha_{0}=-\beta,\\
\alpha_{1}=-a_{2,3}+a_{4,7}+a_{5,6}=1\\
\alpha_{2}=  a_{1,3}+a_{4,6}-a_{5,7}=0\\
\alpha_{3}= -a_{1,2}+a_{4,5}+a_{6,7}=0\\
\alpha_{4}= a_{1,7}+a_{2,6}+a_{3,5}=0\\
\alpha_{5}= a_{1,6}-a_{2,7}-a_{3,4}=0\\
\alpha_{6}= -a_{1,5}-a_{2,4} + a_{3,7}=0\\
\alpha_{7}= -a_{1,4}+ a_{2,5}-a_{3,6}=0
\end{cases}\qquad
\begin{cases}
\tilde\alpha_{0}=\beta,\\
\tilde\alpha_{1}=a_{2,3}-a_{4,7}-a_{5,6}=1\\
\tilde\alpha_{2}=-  a_{1,3}-a_{4,6}+a_{5,7}=0\\
\tilde\alpha_{3}= +a_{1,2}-a_{4,5}-a_{6,7}=0\\
\tilde\alpha_{4}= -a_{1,7}-a_{2,6}-a_{3,5}=0\\
\tilde\alpha_{5}= -a_{1,6}+a_{2,7}+a_{3,4}=0\\
\tilde\alpha_{6}= +a_{1,5}+a_{2,4} - a_{3,7}=0\\
\tilde\alpha_{7}= a_{1,4}-a_{2,5}+a_{3,6}=0.
\end{cases}
\end{equation}
The system has no solution since $\alpha_1+\tilde\alpha_1=0$ and in the same time $\alpha_1+\tilde\alpha_1=2$, which is a contradiction.
\end{proof}
\begin{remark}
Note that the same result we get for any choice of initial vector, where $x_0=(u,u)$ have equal coordinates in the bases of $\mathfrak v_{\min}^-$ and $\mathfrak v_{\min}^+$. This is due to the fact that action of $A$ on $\mathfrak v_{\min}^-$ and $\mathfrak v_{\min}^+$ will have opposite sign due to~\eqref{eq:1=2} and~\eqref{eq:1=2--}.
\end{remark}

\subsection{Pseudo-Riemannian  $H$-type nilmanifold ${N}_{3,4}\Big(\big(\bigoplus^p\mathfrak v_{\min}^{+}\big)\oplus\big(\bigoplus^q \mathfrak v_{\min}^{-}\big)\Big)$}\label{sec:34+-old}

We denote by ${N}_{3,4}\Big(\big(\bigoplus^p\mathfrak v_{\min}^{+}\big)\oplus\big(\bigoplus^q \mathfrak v_{\min}^{-}\big)\Big)$ the pseudo $H$-type nilmanifold,
built over an arbitrary number of non-equivalent irreducible Clifford modules. Using Proposition~\ref{co_totg_2}, Theorem~\ref{th:34++}, and  Theorem~\ref{th:34+-}
one immediately obtains the proof of the following statement.
\begin{theorem}\label{th:34general}
The pseudo-Riemannian $H$-type nilmanifold ${N}_{3,4}\Big(\big(\bigoplus^p\mathfrak v_{\min}^{+}\big)\oplus\big(\bigoplus^q \mathfrak v_{\min}^{-}\big)\Big)$,
where $p+q >1$,
is not a geodesic orbit manifold.
\end{theorem}

\section{Acknowledgment}
The first author was partly supported by JSPS Kakenhi No. 24K06784 and
MEXT Promotion of Distinctive Joint Research Center Program JPMXP0723833165 and Osaka Metropolitan University Strategic Research Promotion Project (Development of International Research Hubs). The second author was partially supported by the Norwegian Research Council through the project "MatNorVeg: Algebra-Geometry-Analysis with applications \#360114".

\vspace{10mm}

\bibliographystyle{amsunsrt}

\end{document}